\documentclass[11pt,a4paper,reqno]{amsart} 
\usepackage{amsmath}
\usepackage{amssymb}
\usepackage{euscript}
\usepackage{mathrsfs}
\usepackage{mathabx}
\usepackage[colorlinks=true,linktocpage=true,pagebackref=false,citecolor=black,linkcolor=black]{hyperref}
\usepackage{caption}
\usepackage{pgfplots}
\usepackage{tikz}
\usepackage{tkz-fct}

\pgfplotsset{soldot/.style={color=black,only marks,mark=*}} \pgfplotsset{holdot/.style={color=black,fill=white,only marks,mark=*}}

\newtheorem{thm}{Theorem}[section]
\newtheorem{prop}[thm]{Proposition}
\newtheorem{cor}[thm]{Corollary}
\newtheorem{lem}[thm]{Lemma}

\theoremstyle{definition}

\theoremstyle{remark}
\newtheorem{example}[thm]{Example}

\newtheorem{remark}[thm]{Remark}
\newtheorem{remarks}[thm]{Remarks}
\newtheorem{remexs}[thm]{Remarks and Examples}

{\em}

{\em}

\newcounter{substep}
\def\thesubstep{\arabic{substep}}

{\em}

\newcounter{subsubstep}
\def\thesubsubstep{\arabic{subsubstep}}

{\em}

\newcommand{\K}{{\mathbb K}} \newcommand{\N}{{\mathbb N}}
 \newcommand{\R}{{\mathbb R}}
 \newcommand{\C}{{\mathbb C}}

\newcommand{\sph}{{\mathbb S}} 
\newcommand{\E}{{\mathbb E}} \newcommand{\PP}{{\mathbb P}}

\newcommand{\Bb}{{\EuScript B}}

\newcommand{\Sing}{\operatorname{Sing}}

\newcommand{\cl}{\operatorname{Cl}}

\newcommand{\id}{\operatorname{id}}

\newcommand{\im}{\operatorname{im}}

\newcommand{\zar}{\operatorname{zar}}

\newcommand{\x}{{\tt x}} \newcommand{\y}{{\tt y}}
\newcommand{\z}{{\tt z}} \renewcommand{\t}{{\tt t}}

\newcommand{\ol}{\overline}

\newcommand{\veps}{\varepsilon}

\numberwithin{equation}{section}

\title[On the one dimensional polynomial, regular and regulous images]{On the one dimensional polynomial, regular and regulous images of closed balls and spheres}

\author{Jos\'e F. Fernando}
\address{Departamento de \'Algebra, Facultad de Ciencias Matem\'aticas, Universidad Complutense de Madrid, 28040 MADRID (SPAIN)}
\email{josefer@mat.ucm.es}
\thanks{The author is supported by Spanish STRANO PID2021-122752NB-I00.}

\begin{document}

\begin{abstract}
We present a full geometric characterization of the $1$-dimensional (semi\-algebraic) images $S$ of either $n$-dimensional closed balls $\ol{\Bb}_n\subset\R^n$ or $n$-dimensional spheres $\sph^n\subset\R^{n+1}$ under polynomial, regular and regulous maps for some $n\geq1$. In all the previous cases one can find a new polynomial, regular or regulous map with domain either $\ol{\Bb}_1:=[-1,1]$ or $\sph^1$ such that $S$ is the image under such map of either $\ol{\Bb}_1:=[-1,1]$ or $\sph^1$. As a byproduct, we provide a full characterization of the images of $\sph^1\subset\C\equiv\R^2$ under Laurent polynomials $f\in\C[\z,\z^{-1}]$, taking advantage of some previous works of Kovalev-Yang and Wilmshurst. We also alternatively prove that all polynomial maps $\sph^k\to\sph^1$ are constant if $k\geq2$.
\end{abstract}

\date{11/05/2025}
\maketitle

\section{Introduction}\label{s1}
A map $f:=(f_1,\ldots,f_m):\R^n\to\R^m$ is a {\em polynomial map} if each of its components $f_i\in\R[\x]:=\R[\x_1,\ldots,\x_n]$ is a polynomial. Let $T\subset\R^n$ and $S\subset\R^m$. We say that $S$ is a {\em polynomial image of $T$} if there exists a polynomial map $f:\R^n\to\R^m$ such that $S=f(T)$. A rational map $f:=(f_1,\ldots,f_m):\R^n\dasharrow\R^m$ is a {\em regular map on $T\subset\R^n$} if each component $f_i\in\R(\x_1,\ldots,\x_n)$ is a rational function, that is, each $f_i:=\frac{g_i}{h_{i}}$ is a quotient of polynomials, and the zero set of $h_{i}$ does not meet $T$. A subset $S$ of $\R^m$ is a {\em regular image of} $T$ if $S=f(T)$ for some rational map $f:\R^n\dasharrow\R^m$ that is regular on $T$. A rational map $f:=(f_1,\ldots,f_m):\R^n\dasharrow\R^m$ is a {\em regulous map on $T\subset\R^n$} if it extends to a continuous function on $T$ and the complement of the set of poles of $f$ meets $T$ in a dense subset of $T$. A subset $S$ of $\R^m$ is a {\em regulous image of} $T$ if $S=f(T)$ for some rational map $f:\R^n\dasharrow\R^m$ that is regulous on $T$. 

A {\em semialgebraic subset} $S$ of $\R^n$ is a set that admits a description as a finite boolean combination of polynomial equalities and inequalities. By \em elimination of quantifiers \em $S$ is semialgebraic if it has a description by a first order formula \em possibly with quantifiers\em. Such a freedom provides semialgebraic descriptions for topological operations. For instance: interiors, closures, borders, connected components of semialgebraic sets are again semialgebraic sets. A map $f:T\to S$ between two semialgebraic sets $T\subset\R^n$ and $S\subset\R^m$ is {\em semialgebraic} if its graph is a semialgebraic set.

A celebrated Theorem of Tarski-Seidenberg \cite[Thm.1.4]{bcr} states that the image of a semialgebraic set $T\subset\R^n$ under a semialgebraic map $f:T\rightarrow\R^m$ (which include the case of polynomial, regular and regulous maps on $T$) is a semialgebraic subset $S$ of $\R^m$. In an \em Oberwolfach \em week \cite{g} Gamboa proposed to characterize the (semialgebraic) sets of $\R^m$ that are polynomial images of $\R^n$ for some $n\geq1$. During the last 25 years we have approached this problem and obtained several results in two directions:
\begin{itemize}
\item{\em General properties.} We have found conditions \cite{fe1,fg2,fu1,u1} that a semialgebraic subset must satisfy to be either a polynomial, regular or regulous image of $\R^m$. The most remarkable one states that the set of points at infinity of a polynomial image of $\R^m$ is connected \cite{fu1}. The $1$-dimensional polynomial images of $\R^n$ were fully described in \cite{fe1}. In \cite[Thm.17]{ffqu} we proved the equality between the family of regular images of $\R^2$ and the family of regulous images of $\R^2$. 
\item{\em Representation of semialgebraic sets as polynomial or regular images of $\R^n$.} We have performed constructions to represent as either polynomial or regular images of $\R^n$ semialgebraic sets that can be described by linear equalities and inequalities. In \cite{fe1,fg1,fgu1,fgu2,fgu3,fgu4,fgu5,fu2,fu3,fu4,fu5,u2} we have analyzed the cases of convex polyhedra and their interiors, together with their respective complements and we have provided a full answer \cite[Table 1]{fu5}.
\end{itemize}

In \cite{kps} Kubjas-Parrillo-Sturmfels proposed to describe explicitely the two dimensional images of $\ol{\Bb}_3$ under a polynomial image $f:\R^3\to\R^2$. We have generalized the previous problem and proposed in \cite{fu6} to determine the semialgebraic subsets of $\R^m$ that are images of either an $n$-dimensional closed ball $\ol{\Bb}_n\subset\R^n$ (of center the origin and radius $1$) or an $n$-dimensional sphere $\sph^n\subset\R^{n+1}$ (of center the origin and radius $1$) mainly under polynomial maps (but also under regular maps). We have obtained full information for the case of unions of finitely many convex polyhedra that provide semialgebraic sets connected by analytic paths \cite{fe3,fu6}. We have also treated in \cite{fu6} more demanding cases, but we feel far from obtaining a full answer for semialgebraic sets of arbitrary dimension.

The class of semialgebraic maps with more tools to attack this type of problems is the Nash category. This is the only case for which we have a full geometric characterization for the images under Nash maps of affine spaces, closed balls and spheres \cite{fe2,cf1,cf2}. Recall here that a \em Nash function \em on an open semialgebraic subset $U\subset\R^m$ is an analytic function on $U$ that satisfies a non-trivial polynomial equation, that is, there exists a non-zero $P\in\R[\x,\y]$ such that $P(x,f(x))=0$ for all $x\in U$. If $S\subset\R^m$ is a semialgebraic set, the ring ${\mathcal N}(S)$ of {\em Nash functions on $S$} is the collection of all functions on $S$ that admits a Nash extension to an open semialgebraic neighborhood $U$ of $S$ in $\R^m$ and it is endowed with the usual sum and product (for further details see \cite{fg3}).

The interest of polynomial, regular, regulous or Nash images of affine spaces, closed balls or spheres arises because there are several problems in Real Algebraic Geometry that for such images can be reduced to analyze them on the corresponding models: affine spaces, closed balls or spheres \cite{fg1,fg2,fgu1,fu2}. Examples of such problems are:
\begin{itemize}
\item optimization of polynomial, regular, regulous or Nash functions on $S$ (see also \cite{th}), 
\item characterization of the polynomial, regular, regulous or Nash functions that are positive semidefinite on $S$ (Hilbert's 17th problem and Positivstellensatz).
\item Constructing Nash paths on semialgebraic sets connected by analytic paths \cite{fe3}.
\end{itemize}

\subsection{Invariants}
Consider the families of models of semialgebraic sets: ${\mathfrak A}:=\{\R^n:\ n\geq1\}$ (affine spaces), ${\mathfrak B}:=\{\ol{\Bb}_n:\ n\geq1\}$ (closed balls) and ${\mathfrak S}:=\{\sph_n:\ n\geq1\}$ (spheres). A semialgebraic set $S\subset\R^m$ is a polynomial image of an affine space (resp. a closed ball or a sphere) if there exist an element $\R^n\in{\mathfrak A}$ (resp. $\ol{\Bb}_n\in{\mathfrak B}$ or $\sph^{n-1}\in{\mathfrak S}$) and a polynomial map defined on $\R^n$ such that $f(\R^n)=S$ (resp. $f(\ol{\Bb}_n)=S$ or $f(\sph^{n-1})=S$). The same definitions can be proposed for regular, regulous and Nash maps in the obvious way.

Let ${\mathfrak E}$ be either ${\mathfrak A}$, ${\mathfrak B}$ or ${\mathfrak S}$ and define for a set $S\subset\R^m$ the following invariants:
\begin{align*}
{\rm p}_{{\mathfrak E}}(S)&:=\inf\{p\geq 1:\ \text{$S$ is a polynomial image of $E\in{\mathfrak E}$ and $\dim(E)=p$}\},\\
{\rm r}_{{\mathfrak E}}(S)&:=\inf\{r\geq 1:\ \text{$S$ is a regular image of $E\in{\mathfrak E}$ and $\dim(E)=r$}\},\\
{\rm rs}_{{\mathfrak E}}(S)&:=\inf\{r\geq 1:\ \text{$S$ is a regulous image of $E\in{\mathfrak E}$ and $\dim(E)=r$}\},\\
{\rm n}_{{\mathfrak E}}(S)&:=\inf\{n\geq 1:\ \text{$S$ is a Nash image of $E\in{\mathfrak E}$ and $\dim(E)=n$}\}.
\end{align*}
In case a subset $A\subset\N$ is empty, we write $\inf(A):=+\infty$. If one of the previous invariant values $+\infty$, then $S$ is not an image of the corresponding type of semialgebraic maps. We have the following initial inequalities:
$$
\max\{{\rm rs}_{{\mathfrak E}}(S),{\rm n}_{{\mathfrak E}}(S)\}\leq{\rm r}_{{\mathfrak E}}(S)\leq{\rm p}_{{\mathfrak E}}(S)
$$
for each $S\subset\R^m$ and each ${\mathfrak E}\in\{{\mathfrak A},{\mathfrak B},{\mathfrak S}\}$. If any of the previous invariants is finite, then $S$ is by Tarski-Seidenberg Theorem \cite[Thm.1.4]{bcr} a semialgebraic set and by \cite[Thm.2.8.8]{bcr} the dimension $\dim(S)$ of $S$ is less than or equal to any of them. 

The closed ball $\ol{\Bb}_n$ is the projection of the sphere $\sph^n$ and $\sph^n$ is a regular image of $\ol{\Bb}_n$ (see \cite[Cor.2.9 \& Lem.A.4]{fu6}). In Example \ref{s1rp1} we recall an explicit regular map $f:\R\to\R^2$ such that $f(\ol{\Bb}_1)=\sph^1$. In addition, the closed ball $\ol{\Bb}_n$ (and consequently the sphere $\sph^n$) is a regular image of $\R^n$ by \cite[Lem.3.1]{fe1} and \cite[Cor.2.9 \& Lem.2.10]{fu6}. Obviously, as both $\ol{\Bb}_n$ and $\sph^n$ are compact sets, $\R^n$ is neither a polynomial, regular, regulous nor a Nash image of either $\ol{\Bb}_n$ or $\sph^n$. In Lemma \ref{boundedk} we show that the image of a compact subset of $\R^n$ with non-empty interior under a polynomial map cannot be a compact algebraic subset of $\R^m$ of dimension $\geq1$. In particular, a sphere $\sph^m$ cannot be the image under a polynomial map of any closed ball $\ol{\Bb}_n$. In addition, polynomial images of $\R^n$ are either unbounded or a singleton \cite[Rem.1.3(3)]{fg1}. We deduce the following extra relations between the invariants: 
$$
\begin{cases}
{\rm r}_{\mathfrak A}(S)\leq{\rm r}_{\mathfrak B}(S)={\rm r}_{\mathfrak S}(S),\\
{\rm rs}_{\mathfrak A}(S)\leq{\rm rs}_{\mathfrak B}(S)={\rm rs}_{\mathfrak S}(S),\\
{\rm rs}_{\mathfrak A}(S)\leq2\ \Longrightarrow\ {\rm rs}_{\mathfrak A}(S)={\rm r}_{\mathfrak A}(S)\ \text{by \cite[Thm.1.7]{ffqu} (see Remark \ref{nor1p2}(ii) below)},\\
{\rm p}_{\mathfrak S}(S)\leq{\rm p}_{\mathfrak B}(S),\\
{\rm p}_{\mathfrak S}(S)<+\infty\ \text{or}\ {\rm p}_{\mathfrak B}(S)<+\infty,\ \text{and}\ \text{$S$ is not a singleton}\Longrightarrow\ {\rm p}_{\mathfrak A}(S)=+\infty,\\
{\rm p}_{\mathfrak A}(S)<+\infty\ \text{and}\ \text{$S$ is not a singleton} \Longrightarrow\ {\rm p}_{\mathfrak B}(S)=+\infty,\ {\rm p}_{\mathfrak S}(S)=+\infty.
\end{cases}
$$ 
The invariants ${\rm n}_{{\mathfrak A}}(S)$, ${\rm n}_{{\mathfrak B}}(S)$ and ${\rm n}_{{\mathfrak S}}(S)$ only take the values $\dim(S)$ (if $S$ is a Nash image) or $+\infty$ (if $S$ is not a Nash image) and have been computed in \cite{fe1} and \cite{cf1,cf2} for each semialgebraic set $S\subset\R^m$. It holds $\dim(S)\leq{\rm n}_{{\mathfrak A}}(S)\leq{\rm n}_{{\mathfrak B}}(S)={\rm n}_{{\mathfrak S}}(S)$ for each semialgebraic set $S\subset\R^m$.

As we have already pointed out in \cite{fg1}, there are some straightforward properties that a regular image $S\subset\R^m$ must satisfy: {\em it has to be pure dimensional, connected, semialgebraic and its Zariski closure has to be irreducible}. Furthermore, $S$ must be by \cite[(3.1) (iv)]{fg3} \em irreducible \em in the sense that its ring ${\mathcal N}(S)$ of Nash functions on $S$ is an integral domain.

{\Small\begin{table}[!ht]
\begin{center}
{\setlength{\arrayrulewidth}{.5pt}
\renewcommand*{\arraystretch}{1.5}
\begin{tabular}{|c|c|c|c|c|c|c|c|c|}
\hline
$S$&$\R$ or $[0,+\infty)$&$(0,+\infty)$&$[0,1)$&$(0,1)$&$[0,1]$&$\sph^1$&Parabola&Non-rational curves\\
\hline
\hline
${\rm p}_{\mathfrak A}(S)$&$1$&$2$&$+\infty$&$+\infty$&$+\infty$&$+\infty$&$1$&$+\infty$\\
\hline
${\rm r}_{\mathfrak A}(S)$&$1$&$2$&$1$&$2$&$1$&$1$&$1$&$+\infty$\\
\hline
${\rm rs}_{\mathfrak A}(S)$&$1$&$2$&$1$&$2$&$1$&$1$&$1$&$+\infty$\\
\hline
${\rm n}_{\mathfrak A}(S)$&$1$&$1$&$1$&$1$&$1$&$1$&$1$&$1$ (I) or $+\infty$ (O)\\
\hline
\hline
${\rm p}_{\mathfrak B}(S)$&$+\infty$&$+\infty$&$+\infty$&$+\infty$&$1$&$+\infty$&$+\infty$&$+\infty$\\
\hline
${\rm r}_{\mathfrak B}(S)$&$+\infty$&$+\infty$&$+\infty$&$+\infty$&$1$&$1$&$+\infty$&$+\infty$\\
\hline
${\rm rs}_{\mathfrak B}(S)$&$+\infty$&$+\infty$&$+\infty$&$+\infty$&$1$&$1$&$+\infty$&$+\infty$\\
\hline
${\rm n}_{\mathfrak B}(S)$&$+\infty$&$+\infty$&$+\infty$&$+\infty$&$1$&$1$&$+\infty$&$1$ (C, I) or $+\infty$ (O)\\
\hline
\hline
${\rm p}_{\mathfrak S}(S)$&$+\infty$&$+\infty$&$+\infty$&$+\infty$&$1$&$1$&$+\infty$&$+\infty$\\
\hline
${\rm r}_{\mathfrak S}(S)$&$+\infty$&$+\infty$&$+\infty$&$+\infty$&$1$&$1$&$+\infty$&$+\infty$\\
\hline
${\rm rs}_{\mathfrak S}(S)$&$+\infty$&$+\infty$&$+\infty$&$+\infty$&$1$&$1$&$+\infty$&$+\infty$\\
\hline
${\rm n}_{\mathfrak S}(S)$&$+\infty$&$+\infty$&$+\infty$&$+\infty$&$1$&$1$&$+\infty$&$1$ (C, I) or $+\infty$ (O)\\
\hline
\end{tabular}}
\end{center}
\caption{Notations: C:=Compact, I:=Irreducible, O:=Otherwise\label{table1}}
\end{table}}

\subsection{The one dimensional case.}
In this work we focus our attention on the one dimensional case and present a full geometric characterization of the polynomial, regular and regulous one dimensional images of closed balls and spheres. In fact, we compute the exact value of the invariants ${\rm p}_{\mathfrak E}$, ${\rm r}_{\mathfrak E}$ and ${\rm rs}_{\mathfrak E}$ for all of them, where ${\mathfrak E}={\mathfrak B},{\mathfrak S}$. We will see in this work that in the one dimensional case the only possible values for the invariants ${\rm p}_{\mathfrak E}$, ${\rm r}_{\mathfrak E}$ and ${\rm rs}_{\mathfrak E}$ (where ${\mathfrak E}={\mathfrak B},{\mathfrak S}$) are either $1$ or $+\infty$. In {\sc Table} \ref{table1} we illustrate the situation with several examples and compare the invariants ${\rm p}_{\mathfrak E}$, ${\rm r}_{\mathfrak E}$ and ${\rm rs}_{\mathfrak E}$ (where ${\mathfrak E}={\mathfrak B},{\mathfrak S}$) with the invariants ${\rm p}_{\mathfrak A}$, ${\rm r}_{\mathfrak A}$, ${\rm rs}_{\mathfrak A}$ and ${\mathfrak n}_{\mathfrak E}$ (where ${\mathfrak E}={\mathfrak A},{\mathfrak B},{\mathfrak S}$), which were mainly computed in \cite{cf1,cf2,fe1,fe2,ffqu}. 

\subsection{Notations and terminology}
Before stating our main results whose proofs are developed in Section \ref{s3}, after the preparatory work of Section \ref{s2}, we recall some preliminary standard notations and terminology. We write $\K$ to refer indistinctly to $\R$ or $\C$ and denote the hyperplane at infinity of the projective space $\K\PP^n$ with $\mathsf{H}^n_\infty(\K):=\{\x_0=0\}$. The projective space $\K\PP^n$ contains $\K^n$ as the set $\K\PP^n\setminus\mathsf{H}^n_\infty(\K)=\{\x_0=1\}$. If $n=1$, the point of infinity of the projective line $\K\PP^1$ is ${[0:1]}$.

For each $n\geq1$ denote the complex conjugation with
$$
\sigma_n:\C\PP^n\to\C\PP^n,\ z:=[z_0:z_1:\cdots:z_n]\mapsto\ol{z}:=[\ol{z_0}:\ol{z_1}:\cdots:\ol{z_n}].
$$
Clearly, $\R\PP^n$ is the set of fixed points of $\sigma_n$. A set $A\subset\C\PP^n$ is called \em invariant \em if $\sigma_n(A)=A$. It is well-known that if $Z\subset\C\PP^n$ is an invariant non-singular (complex) projective variety, then $Z\cap\R\PP^n$ is a non-singular (real) projective variety. We also say that a rational map $h:\C\PP^n\dashrightarrow\C\PP^m$ is \em invariant \em if $h\circ\sigma_n=\sigma_m\circ h$. Of course, $h$ is invariant if its components can be chosen as homogeneous polynomials with real coefficients, so it provides by restriction a real rational map $h|_{\R\PP^n}:\R\PP^n\dashrightarrow\R\PP^m$.

Given a semialgebraic set $S\subset\R^m\subset\R\PP^m\subset\C\PP^m$, we denote its Zariski closure in $\K\PP^m$ with $\cl_{\K\PP^m}^{\zar}(S)$. Obviously, $\cl_{\C\PP^m}^{\zar}(S)\cap\R\PP^m=\cl_{\R\PP^m}^{\zar}(S)$ and $\cl^{\zar}(S)=\cl_{\R\PP^m}^{\zar}(S)\cap\R^m$ is the \em Zariski closure of $S$ in $\R^m$\em. Observe that $\cl_{\C\PP^m}^{\zar}(S)$ is an invariant algebraic set. In addition, $\cl_{\K\PP^m}(S)$ denotes the closure of $S$ in $\K\PP^m$ with respect to the quotient topology of $\K\PP^m$ induced by the canonical map $\pi:\K^{m+1}\setminus\{0\}\to\K\PP^m,\ x:=(x_0,x_1,\ldots,x_n)\mapsto[x]:=[x_0:x_1:\cdots:x_n]$. We endow the linear space $\K^{m+1}$ for $\K=\R$ or $\C$ with the Euclidean topology (in the case $\K=\R$ it is induced by the Euclidean norm, whereas in the case $\K=\C$ it is induced by the norm associated to its usual Hermitian inner product). The projective spaces $\K\PP^m$ (endowed with the previous topology) can be embedded as real algebraic submanifolds of $\R^M$ for some positive integer $M$ large enough \cite[\S.3.4.1 \& Prop.3.4.6]{bcr}.

A \em complex rational curve \em is the image of $\C\PP^1$ under a birational map, which is by \cite[Prop.(7.1)]{m1} in addition regular, because $\C\PP^1$ does not have singular points (see also Lemma \ref{curve}). We denote the {\em set of points of} the semialgebraic set {$S$ that have local dimension $k$} with $S_{(k)}$, which is a semialgebraic subset of $S$. If $k=\dim(S)$, then $S_{(k)}$ is in addition closed in $S$. A \em real rational curve \em is a real projective irreducible curve $C$ such that $C_{(1)}$ is the image of $\R\PP^1$ under a birational map, which by Lemma \ref{curve} is in addition a regular map.

We also deal with the {\em irreducibility of analytic set germs}. A set germ $X_p$ of $\K\PP^n_p$ is {\em analytic} if there exist finitely many analytic functions $f_1,\ldots,f_s$ on a neighborhood $U$ of $p$ (for instance, polynomial or regular on $\{p\}$) such that $X_p$ is the set germ at $p$ of the common zero set of $f_1,\ldots,f_s$. An analytic set germ $X_p$ is {\em reducible} if there exist analytic set germs $X_{1,p}$ and $X_{2,p}$ such that $X_p=X_{1,p}\cup X_{2,p}$ and $X_p\neq X_{i,p}$ for $i=1,2$. Otherwise, we say that $X_p$ is {\em irreducible}. The {\em analytic closure} of a set germ $S_p$ of $\K\PP^n_p$ is the smallest analytic set germ $X_p$ that contains $S_p$.

\subsubsection{State of the art.} 
We recall the geometric characterization of the $1$-dimensional polynomial images of affine spaces proposed in \cite[Thm.1.1 \& Prop.1.2]{fe1} (see also \cite[Prop.2.1, Cor.2.2]{fg2}) and the description of those with ${\rm p}_{\mathfrak A}=1$.

\begin{thm}\label{dim1p}
Let $S\subset\R^m$ be a $1$-dimensional semialgebraic set. The following conditions are equivalent:
\begin{itemize}
\item[(i)] ${\rm p}_{\mathfrak A}(S)\leq2$.
\item[(ii)] ${\rm p}_{\mathfrak A}(S)<+\infty$.
\item[(iii)] $S$ is irreducible, unbounded and $\cl_{\C\PP^m}^{\zar}(S)$ is an invariant rational curve such that the set of points at infinity $\cl_{\C\PP^m}^{\zar}(S)\cap\mathsf{H}^m_\infty(\C)$ is a singleton $\{p\}$ and the analytic set germ $\cl_{\C\PP^m}^{\zar}(S)_p$ is irreducible.
\end{itemize}
We also have: ${\rm p}_{\mathfrak A}(S)=1$ if and only if ${\rm p}_{\mathfrak A}(S)<+\infty$ and $S$ is closed in $\R^m$.
\end{thm}

The counterpart of the previous results in the regular setting, already proved in \cite[Thm.1.3 \& Prop.1.4]{fe1}, consists of the full geometric characterization of the $1$-dimensional regular images of affine spaces and the description of those with ${\rm r}_{\mathfrak A}=1$.

\begin{thm}\label{dim1r}
Let $S\subset\R^m$ be a $1$-dimensional semialgebraic set. The following conditions are equivalent:
\begin{itemize}
\item[(i)] ${\rm r}_{\mathfrak A}(S)\leq2$.
\item[(ii)] ${\rm r}_{\mathfrak A}(S)<+\infty$.
\item[(iii)] $S$ is irreducible and $\cl_{\R\PP^m}^{\zar}(S)$ is a rational curve.
\end{itemize}
We also have: ${\rm r}_{\mathfrak A}(S)=1$ if and only if ${\rm r}_{\mathfrak A}(S)<+\infty$ and either 
\begin{itemize}
\item[(1)] $\cl_{\R\PP^m}(S)=S$ or 
\item[(2)] $\cl_{\R\PP^m}(S)\setminus S=\{p\}$ is a singleton and the analytic closure of the set germ $S_p$ is irreducible. 
\end{itemize}
\end{thm}

\begin{remarks}\label{nor1p2}
(i) There is no $1$-dimensional semialgebraic set $S\subset\R^m$ with ${\rm p}_{\mathfrak A}(S)=2$ and ${\rm r}_{\mathfrak A}(S)=1$, see \cite[Cor.1.5]{fe1}.

(ii) Let $S\subset\R^m$ be a semialgebraic set of dimension $1$. We claim: ${\rm r}_{\mathfrak A}(S)={\rm rs}_{\mathfrak A}(S)$.

As regular functions are regulous functions, ${\rm rs}_{\mathfrak A}(S)\leq{\rm r}_{\mathfrak A}(S)$. Let us check: {\em If ${\rm rs}_{\mathfrak A}(S)<+\infty$, then ${\rm r}_{\mathfrak A}(S)<+\infty$.} This implies the equivalence: {\em ${\rm rs}_{\mathfrak A}(S)<+\infty$ if and only if ${\rm r}_{\mathfrak A}(S)<+\infty$}. By Theorem \ref{dim1r} we have to prove: {\em $\cl_{\R\PP^m}^{\zar}(S)$ is a rational curve and $S$ is irreducible}. 

Let $n\geq1$ and $f:\R^n\to\R^m$ be a regulous map such that $f(\R^n)=S$. By \cite[Lem.2.2(i) \& Lem.2.3(i)]{fe1} there exists a rational function $g\in\R(\x_1,\ldots,\x_n)$ and a regular map $h:\R\to\R^m$ such that $f=h\circ g$. Then $S\subset T:=\im(h)$. By \cite[Lem.2.2(ii)]{fe1} the Zariski closure $\cl_{\R\PP^m}^{\zar}(T)$ is a rational curve. As $S$ is $1$-dimensional and $\cl_{\R\PP^m}^{\zar}(T)$ is irreducible, $\cl_{\R\PP^m}^{\zar}(S)=\cl_{\R\PP^m}^{\zar}(T)$ is a rational curve.

To prove that $S$ is irreducible it is enough to show by \cite[Main Thm.1.4 \& Lem.7.3]{fe2} that $S$ is connected by analytic paths. Pick two different points $y_1,y_2\in S$ and let $x_1,x_2\in\R^n$ be such that $f(x_i)=y_i$ for $i=1,2$. Consider the line $L\subset\R^n$ that passes through $x_1,x_2$ and let $\varphi:\R\to\R^n$ be an affine parameterization of $L$ such that there exist values $t_1<t_2$ in $\R$ satisfying $\varphi(t_i)=x_i$ for $i=1,2$. The map $f\circ\varphi:\R\to\R^m$ is regulous, so by \cite[Cor.3.6]{fhmm} $f$ is a regular map and, consequently, it is an analytic map. Thus, $\alpha:=(f\circ\varphi)|_{[t_1,t_2]}:[t_1,t_2]\to S$ is an analytic path that connects $y_1$ and $y_2$, so $S$ is connected by analytic paths. 

Suppose next that ${\rm r}_{\mathfrak A}(S)<+\infty$. By Theorem \ref{dim1r} we have ${\rm rs}_{\mathfrak A}(S)\leq{\rm r}_{\mathfrak A}(S)\leq2$. If ${\rm rs}_{\mathfrak A}(S)=2$, then $2={\rm rs}_{\mathfrak A}(S)\leq{\rm r}_{\mathfrak A}(S)\leq2$, so ${\rm rs}_{\mathfrak A}(S)={\rm r}_{\mathfrak A}(S)=2$. Assume next ${\rm rs}_{\mathfrak A}(S)=1$. As the regulous maps on $\R$ coincide by \cite[Cor.3.6]{fhmm} with the regular maps on $\R$, we deduce ${\rm r}_{\mathfrak A}(S)={\rm rs}_{\mathfrak A}(S)=1$.\hfill$\sqbullet$ 
\end{remarks}

In the Nash case we have the following two conclusive results proved in \cite{fe2,cf1,cf2}:
\begin{thm}[{\cite[Prop.1.6]{fe2}}]
Let $S\subset\R^m$ be a $1$-dimensional semialgebraic set. The following conditions are equivalent:
\begin{itemize}
\item[(i)] ${\rm n}_{\mathfrak A}(S)=1$.
\item[(ii)] ${\rm n}_{\mathfrak A}(S)<+\infty$.
\item[(iii)] $S$ is irreducible.
\end{itemize}
\end{thm}

\begin{thm}[{\cite[Prop.1.20]{cf1}}, {\cite[Thm.1.10]{cf2}}]
Let $S\subset\R^m$ be a $1$-dimensional semialgebraic set. The following conditions are equivalent:
\begin{itemize}
\item[(i)] ${\rm n}_{\mathfrak B}(S)=1$.
\item[(ii)] ${\rm n}_{\mathfrak S}(S)=1$.
\item[(iii)] ${\rm n}_{\mathfrak S}(S)<+\infty$.
\item[(iv)] ${\rm n}_{\mathfrak B}(S)<+\infty$.
\item[(v)] $S$ is irreducible and compact.
\end{itemize}
\end{thm}

\subsection{Main results}
The main results of this article, which will be proved in Section \ref{s3}, are the following. We begin with the invariants corresponding to regular and regulous cases.

\begin{thm}\label{rb1s1}
Let $S\subset\R^m$ be a $1$-dimensional semialgebraic set. The following conditions are equivalent:
\begin{itemize}
\item[(i)] ${\rm r}_{\mathfrak B}(S)=1$.
\item[(ii)] ${\rm r}_{\mathfrak B}(S)<+\infty$.
\item[(iii)] ${\rm rs}_{\mathfrak B}(S)=1$.
\item[(iv)] ${\rm rs}_{\mathfrak B}(S)<+\infty$.
\item[(v)] ${\rm r}_{\mathfrak S}(S)=1$.
\item[(vi)] ${\rm r}_{\mathfrak S}(S)<+\infty$.
\item[(vii)] ${\rm rs}_{\mathfrak S}(S)=1$.
\item[(viii)] ${\rm rs}_{\mathfrak S}(S)<+\infty$.
\item[(ix)] $S$ is irreducible, compact and $\cl_{\R\PP^m}^{\zar}(S)$ is a rational curve.
\end{itemize}
\end{thm}

We present next the results for the polynomial case. We begin with the invariant corresponding to closed balls, which is simpler.

\begin{thm}\label{pb1}
Let $S\subset\R^m$ be a $1$-dimensional semialgebraic set. The following conditions are equivalent:
\begin{itemize}
\item[(i)] ${\rm p}_{\mathfrak B}(S)=1$.
\item[(ii)] ${\rm p}_{\mathfrak B}(S)<+\infty$.
\item[(iii)] $S$ is irreducible, compact and $\cl_{\C\PP^m}^{\zar}(S)$ is an invariant rational curve such that the set of points at infinity $\cl_{\C\PP^m}^{\zar}(S)\cap\mathsf{H}^m_\infty(\C)$ is a singleton $\{p\}\subset\mathsf{H}^m_\infty(\R)$ and the analytic set germ $\cl_{\C\PP^m}^{\zar}(S)_p$ is irreducible.
\end{itemize}
\end{thm}

The case of polynomial images of spheres presents the following peculiarity, which contrast with the polynomial images of affine spaces: {\em The family of polynomial images of the circle $\sph^1$ is larger than the family of the polynomial images of the spheres $\sph^k$ of dimension $k\geq2$.} This statement is deduced from Theorem \ref{ps1} and Proposition \ref{ps2} below. In fact, the reader can check using the quoted two results that $\sph^1$ is an example of a $1$-dimensional semialgebraic set, which is a polynomial image of $\sph^1$, but it is not a polynomial image of $\sph^k$ for each $k\geq2$. ({\em Hint: $\cl_{\C\PP^2}^{\zar}(\sph^1)\cap\mathsf{H}^2_\infty(\C)=\{[0:1:{\tt i}],[0:1:-{\tt i}]\}$}). We denote along this article ${\tt i}:=\sqrt{-1}$.

We characterize next the images of $\sph^1$ in $\R^m$ under polynomial maps.

\begin{thm}\label{ps1}
Let $S\subset\R^m$ be a $1$-dimensional semialgebraic set. The following conditions are equivalent:
\begin{itemize}
\item[(i)] ${\rm p}_{\mathfrak S}(S)=1$.
\item[(ii)] $S$ is irreducible, compact and the Zariski closure $\cl_{\C\PP^m}^{\zar}(S)$ is an invariant rational curve such that one of the following three situations hold:
\begin{itemize}
\item[(1)] $\cl_{\C\PP^m}^{\zar}(S)\cap\mathsf{H}^m_\infty(\C)=\{p\}$ is a singleton (that belongs to $\mathsf{H}^m_\infty(\R)$) and the analytic set germ $\cl_{\C\PP^m}^{\zar}(S)_p$ is irreducible.
\item[(2)] $\cl_{\C\PP^m}^{\zar}(S)\cap\mathsf{H}^m_\infty(\C)=\{p\}$ is a singleton (that belongs to $\mathsf{H}^m_\infty(\R)$), the analytic set germ $\cl_{\C\PP^m}^{\zar}(S)_p$ has exactly two irreducible components that are conjugated, and $S=\cl_{\R\PP^m}^{\zar}(S)_{(1)}$.
\item[(3)] $\cl_{\C\PP^m}^{\zar}(S)\cap\mathsf{H}^m_\infty(\C)=\{q,\ol{q}\}$ (where the points $q,\ol{q}\not\in\mathsf{H}^m_\infty(\R)$), the analytic set germs $\cl_{\C\PP^m}^{\zar}(S)_q$ and $\cl_{\C\PP^m}^{\zar}(S)_{\ol{q}}$ are irreducible and conjugated, and $S=\cl_{\R\PP^m}^{\zar}(S)_{(1)}$.
\end{itemize}
\end{itemize}
\end{thm}
\begin{remexs}\label{infty0}
Let $S\subset\R^m$ be a $1$-dimensional semialgebraic set.

(i) We will prove in Lemma \ref{infty1} that if $\cl_{\C\PP^m}^{\zar}(S)\cap\mathsf{H}^m_\infty(\C)=\{p\}$ and $\cl_{\C\PP^m}^{\zar}(S)_p$ is irreducible, then $\cl_{\C\PP^m}^{\zar}(S)\cap\R^m$ is unbounded (case (ii.1) above).

An example of this situation is $S:=[-1,1]\times\{0\}\subset\R^2$. We have $\cl_{\C\PP^2}^{\zar}(S)=\{\x_2=0\}$, $\cl_{\C\PP^2}^{\zar}(S)\cap\mathsf{H}^2_\infty(\C)=\{p:=[0:1:0]\}$ and $\cl_{\C\PP^2}^{\zar}(S)_p=\{x_2=0\}_p$ is irreducible. 

\begin{center}
\begin{figure}[ht!]
\begin{tikzpicture}[scale=3,domain=-pi:pi,samples=3000,smooth]
\draw[->,thick] (-1.2,0) -- (1.2,0) node[right] {$\x_1$};
\draw[->,thick] (0,-0.6) -- (0,0.6) node[above] {$\x_2$};
\draw (0,0) node {\huge$\bullet$};
\draw[color=black,line width=1.5pt] plot ({cos(\x r)},{cos(\x r)*sin(\x r)});
\end{tikzpicture}
\caption{Gerono's leminiscate\label{fig2}}
\end{figure}
\end{center}
\vspace*{-0.75cm}

(ii) We will prove in Lemma \ref{infty2} that if $\cl_{\C\PP^m}^{\zar}(S)\cap\mathsf{H}^m_\infty(\C)=\{p\}$ is a singleton and the analytic set germ $\cl_{\C\PP^m}^{\zar}(S)_p$ has exactly two irreducible components that are conjugated, then $\cl_{\C\PP^m}^{\zar}(S)\cap\R^m$ is bounded (case (ii.2) above).

An example of this situation is Gerono's leminiscate $S:=\{\x_2^2-\x_1^2+\x_1^4=0\}\subset\R^2$ ({\sc Figure} \ref{fig2}). Then $\cl_{\C\PP^2}^{\zar}(S)=\{\x_0^2(\x_2^2-\x_1^2)+\x_1^4=0\}$, $\cl_{\C\PP^2}^{\zar}(S)\cap\mathsf{H}^2_\infty(\C)=\{p:=[0:0:1]\}$ and $\cl_{\C\PP^2}^{\zar}(S)_p$ has two conjugated irreducible components parameterized by $\alpha:=[{\tt i}\frac{\t^2}{\sqrt{1-\t^2}}:\t:1]$ and $\ol{\alpha}:=[-{\tt i}\frac{\t^2}{\sqrt{1-\t^2}}:\t:1]$. Observe that $\cl_{\C\PP^2}^{\zar}(S)$ is the rational curve parameterized by 
$$
\Pi:\C\PP^1\to\cl_{\C\PP^2}^{\zar}(S),\ [\t_0:\t_1]\mapsto[(\t_0^2+\t_1^2)^2:\t_1^4-\t_0^4:2\t_0\t_1(\t_1^2-\t_0^2)]
$$ 
and $S$ is the image of $\sph^1$ under the polynomial map $f:\R^2\to\R^2,\ (x,y)\mapsto(x,xy)$.

(iii) If $\cl_{\C\PP^m}^{\zar}(S)\cap\mathsf{H}^m_\infty(\C)=\{q,\ol{q}\}$ (with $q\neq\ol{q}$) and the analytic set germs $\cl_{\C\PP^m}^{\zar}(S)_q$ and $\cl_{\C\PP^m}^{\zar}(S)_{\ol{q}}$ are irreducible and conjugated, then $\cl_{\C\PP^m}^{\zar}(S)\cap\R^m$ is bounded (case (ii.3) above).

As $\cl_{\C\PP^m}^{\zar}(S)\cap\mathsf{H}^m_\infty(\C)=\{q,\ol{q}\}$, we have $\cl_{\R\PP^m}^{\zar}(S)\cap\mathsf{H}^m_\infty(\R)=\varnothing$, so $\cl_{\R\PP^m}^{\zar}(S)=\cl_{\R\PP^m}^{\zar}(S)\cap\R^m$ is compact.

An example of this situation is $S:=\{\x_1^2+\x_2^2-1=0\}\subset\R^2$. Then $\cl_{\C\PP^2}^{\zar}(S)=\{\x_1^2+\x_2^2-\x_0^2=0\}$, $\cl_{\C\PP^2}^{\zar}(S)\cap\mathsf{H}^2_\infty(\C)=\{q:=[0:1:{\tt i}],\ol{q}:=[0:1:-{\tt i}]\}$ and the analytic set germs $\cl_{\C\PP^m}^{\zar}(S)_q$ and $\cl_{\C\PP^m}^{\zar}(S)_{\ol{q}}$ are irreducible, non-singular and conjugated. We have used that $\cl_{\C\PP^2}^{\zar}(S)$ is a non-singular invariant (complex) projective algebraic set.
\hfill$\sqbullet$
\end{remexs}

\subsubsection{Images of the unit circumference under Laurent polynomials}
The polynomial images of $\sph^1$ in $\R^2$ coincide with the images of $\sph^1$ under Laurent polynomials $f\in\C[\z,\z^{-1}]$ in one variable $\z$ with coefficients in $\C$. We refer the reader to \cite[Thm.2.1]{ky} (whose proof strongly relies on \cite{wi}) for a result that explores the algebraic structures of such images. This result is not fully conclusive, but it is crucial to analyze situations (ii.2) and (iii.3) of Theorem \ref{ps1}. As a consequence of Theorem \ref{ps1} we provide the full characterization of the images of $\sph^1$ under Laurent polynomials completing the valuable information provided in \cite[Thm.2.1]{ky}.

\begin{cor}\label{ps1c}
Let $S\subset\C\equiv\R^2$ be a $1$-dimensional semialgebraic set. The following conditions are equivalent:
\begin{itemize}
\item[(i)] There exists a Laurent polynomial $f\in\C[\z,\z^{-1}]$ such that $f(\sph^1)=S$.
\item[(ii)] $S$ is irreducible, compact and the Zariski closure $\cl_{\C\PP^2}^{\zar}(S)$ is an invariant rational curve such that one of the following three situations hold:
\begin{itemize}
\item[(1)] $\cl_{\C\PP^2}^{\zar}(S)\cap\mathsf{H}^2_\infty(\C)=\{p\}$ is a singleton (that belongs to $\mathsf{H}^2_\infty(\R)$) and the analytic set germ $\cl_{\C\PP^2}^{\zar}(S)_p$ is irreducible.
\item[(2)] $\cl_{\C\PP^2}^{\zar}(S)\cap\mathsf{H}^2_\infty(\C)=\{p\}$ is a singleton (that belongs to $\mathsf{H}^2_\infty(\R)$), the analytic set germ $\cl_{\C\PP^2}^{\zar}(S)_p$ has exactly two irreducible components that are conjugated, and $S=\cl_{\R\PP^2}^{\zar}(S)_{(1)}$.
\item[(3)] $\cl_{\C\PP^2}^{\zar}(S)\cap\mathsf{H}^2_\infty(\C)=\{q,\ol{q}\}$ (where $q,\ol{q}\not\in\mathsf{H}^2_\infty(\R)$), the analytic set germs $\cl_{\C\PP^2}^{\zar}(S)_q$ and $\cl_{\C\PP^2}^{\zar}(S)_{\ol{q}}$ are irreducible and conjugated, and $S=\cl_{\R\PP^2}^{\zar}(S)_{(1)}$.
\end{itemize}
\end{itemize}
\end{cor}
\begin{remark}
In reference to \cite[Thm.2.1]{ky} observe that in case (ii.1) the Zariski closure $\cl^{\zar}_{\C\PP^2}(S)\cap\R^2$ is unbounded (use Lemma \ref{infty1}), so the difference $\cl_{\R\PP^2}^{\zar}(S)\setminus S$ is an infinite ($1$-dimensional semialgebraic) set (because $S$ is compact), whereas in cases (ii.2) and (ii.3) the Zariski closure $\cl_{\C\PP^2}^{\zar}(S)\cap\R^2$ is bounded (use Lemma \ref{infty2} and Remark \ref{infty0}(iii)) and the difference $\cl_{\R\PP^2}^{\zar}(S)\setminus S=\cl_{\R\PP^2}^{\zar}(S)_{(0)}$ is a finite set (maybe empty).
\hfill$\sqbullet$
\end{remark}

\subsubsection{Images of the unit spheres of higher dimension}
As a consequence of Theorem \ref{pb1} and Proposition \ref{ps2}, if we consider spheres $\sph^k$ for some $k\geq2$ instead of the circumference $\sph^1$, one realizes that the polynomial images of $\sph^k$ coincide with those of a closed interval. By \cite[Thm.2]{w} all polynomial maps $f:\sph^2\to\sph^1$ are constant. Consequently, by \cite[Lem.13.1.1]{bcr} all polynomial maps $f:\sph^k\to\sph^1$ are constant for each $k\geq2$. 

\begin{prop}\label{ps2}
Let $S\subset\R^m$ be a $1$-dimensional semialgebraic set. The following conditions are equivalent:
\begin{itemize}
\item[(i)] $S$ is the image of $\sph^2$ under a polynomial map $f:\R^3\to\R^m$.
\item[(ii)] $S$ is the image of $\sph^k$ under a polynomial map $f:\R^{k+1}\to\R^m$ for some $k\geq2$.
\item[(iii)] $S$ is irreducible, compact, $\cl_{\R\PP^m}^{\zar}(S)\cap\R^m$ is unbounded and $\cl_{\C\PP^m}^{\zar}(S)$ is an invariant rational curve such that the set of points at infinity $\cl_{\C\PP^m}^{\zar}(S)\cap\mathsf{H}^m_\infty(\C)$ is a singleton $\{p\}$ and the analytic set germ $\cl_{\C\PP^m}^{\zar}(S)_p$ is irreducible.
\end{itemize}
\end{prop}

\begin{remarks}
Proposition \ref{ps2} alternatively proves that all polynomial maps $f:\sph^k\to\sph^1$ are constant if $k\geq2$ (see also \cite[Thm.2]{w} and \cite[Lem.13.1.1]{bcr}).
\end{remarks}

\subsection*{Acknowledgements}
This article was written while supervising the related Bachelor's Thesis of Alejandro G\'omez G\'omez entitled ``Im\'agenes polin\'omicas y regulares de dimensión 1''. The author is indebted to S. Schramm for a careful reading of the final version and for the suggestions to refine its redaction. The author also thanks the anonymous referee for very valuable suggestions to correct several inaccuracies and incomplete arguments. These suggestions have notably improved the manuscript and made the final version of this article clearer and more precise. 

\section{Main tools}\label{s2}

In this section we present the main tools used to prove the results proposed in this article. We will use usual concepts of (complex) Algebraic Geometry such as: rational map, regular map, normalization, etc. and refer the reader to \cite{m1,m2,sh1} for further details. We begin proving that compact algebraic sets of dimension $\geq1$ are not images of compact subsets $K\subset\R^n$ with non-empty interior under polynomial maps.

\begin{lem}\label{boundedk}
Let $X\subset\R^m$ be a compact algebraic set and $f:=(f_1,\ldots,f_m):\R^n\to\R^m$ a polynomial map. Let $K\subset\R^n$ be a compact set with non-empty interior in $\R^n$ such that $f(K)\subset X$. Then $f(K)$ is a singleton contained in $X$.
\end{lem}
\begin{proof}
Let $\Bb\subset K$ be a non-empty open ball. We claim: {\em The Zariski closures of both $f(\R^n)$ and $f(\Bb)$ coincide}. 

If $g\in\R[\x_1,\ldots,\x_m]$ satisfies $g(f(\Bb))=0$, then $(g\circ f)|_{\Bb}=0$, so by the Identity Principle $g\circ f=0$ on $\R^n$ and $g(f(\R^n))=0$. Thus, $\cl^{\zar}(f(\R^n))\subset\cl^{\zar}(f(\Bb))\subset\cl^{\zar}(f(\R^n))$, so $\cl^{\zar}(f(\R^n))=\cl^{\zar}(f(\Bb))$.

As $f(\Bb)\subset f(K)\subset X$ and $X$ is an algebraic set, $f(\R^n)\subset\cl^{\zar}(f(\R^n))=\cl^{\zar}(f(\Bb))\subset X$. As $X$ is bounded, we deduce $f(\R^n)$ is bounded, so it is by \cite[Rem.1.3(3)]{fg1} a singleton (contained in $X$), as required. 
\end{proof}

We will use freely along this work the existence of regular maps $f:\R^n\to\R^{n+1}$ such that $f(\ol{\Bb}_n)=\sph^n$ (see \cite[Cor.2.9 \& Lem.A.4]{fu6}). We recall here an explicit regular map $f:\R\to\R^2$ such that $f(\ol{\Bb}_1)=\sph^1$ for the case $n=1$.

\begin{example}\label{s1rp1}
Let us show that $\sph^1$ and $\R\PP^1$ are regular images of $[-1,1]$. Since $\R\PP^1$ is the image of $\sph^1$ via the canonical projection $\pi:\sph^1\to\R\PP^1$, it is enough to prove that $\sph^1$ is a regular image of $[-1,1]$. To that end, we may take for instance the regular map
$$
f:\R\to\sph^1,\ t\mapsto\Big(\Big(\frac{2t}{t^2+1}\Big)^2-\Big(\frac{t^2-1}{t^2+1}\Big)^2,2\Big(\frac{2t}{t^2+1}\Big)\Big(\frac{t^2-1}{t^2+1}\Big)\Big),
$$
which satisfies $f(\ol{\Bb}_1)=f([-1,1])=\sph^1$. The previous map $f$ is the composition of the inverse of the stereographic projection 
$$
\varphi:\R\to\sph^1,\ t\mapsto\Big(\frac{2t}{t^2+1},\frac{t^2-1}{t^2+1}\Big)
$$
of $\sph^1$ from $(0,1)$ with 
$$
g:\C\equiv\R^2\to\C\equiv\R^2,\ z=x+{\tt i}y\equiv(x,y)\mapsto z^2\equiv(x^2-y^2,2xy).
$$
$ $\hfill$\sqbullet$ 
\end{example}

We recall the following useful and well-known fact concerning the regularity of rational maps defined on a non-singular curve \cite[Prop.(7.1)]{m1} that will be used several times. As a straightforward consequence of the following result applied to $Z=\C\PP^1$, the reader can deduce (alternatively to \cite[Prop.3.5]{fhmm}) that each regulous map $f:\R\to\R^m$ is in fact a regular map.

\begin{lem}\label{curve}
Let $Z\subset\C\PP^n$ be a non-singular projective curve and $F:Z\dashrightarrow\C\PP^m$ a rational map. Then $F$ can be (uniquely) extended to a regular map $F':Z\to\C\PP^m$. Moreover, if $Z,F$ are invariant, then also $F'$ is invariant.
\end{lem}

\subsection{Normalization of an algebraic curve}\label{nac} 
A main tool to prove the results of this article will be the normalization of either affine or projective algebraic curves $X$ of either $\C^n$ or $\C\PP^n$. We refer the reader to \cite[Ch.II.\S5]{sh1} and \cite[III.\S9]{m2} for a detailed exposition. Let $X$ be an either affine or projective algebraic curve $X$ of either $\E_\C^n:=\C^n$ or $\C\PP^n$ and denote the set of singular points of $X$ with $\Sing(X)$. The normalization $(\widetilde{X},\Pi)$ of $X$ is a pair constituted by a non-singular algebraic set $\widetilde{X}\subset\E_\C^k$ and a (birational) regular map $\Pi:\widetilde{X}\to X$ such that the restriction $\Pi|_{\widetilde{X}\setminus\Pi^{-1}(\Sing(X))}:\widetilde{X}\setminus\Pi^{-1}(\Sing(X))\to X\setminus\Sing(X)$ is a biregular diffeomorphism. The normalization is unique up to a biregular diffeomorphism. Recall that all fibers of $\Pi:\widetilde{X}\to X$ are finite and if $x\in X$ is a non-singular point, then the fiber of $x$ is a singleton. If $X$ is a complex algebraic curve, the cardinal of the fiber of a point $x\in X$ coincides with the number of irreducible components of the analytic set germ $X_x$. If $\Pi^{-1}(x):=\{z_1,\ldots,z_r\}$, the irreducible components of the analytic set germ $X_x$ are $\Pi(\widetilde{X}_{z_1}),\ldots,\Pi(\widetilde{X}_{z_r})$. 

Denote the complex conjugation of $\E_\C^n$ with $\sigma_n$. If $X$ is an invariant complex algebraic curve, we may assume that both $\widetilde{X}$ and $\Pi$ are also invariant. To prove this, one can construct $(\widetilde{X},\Pi)$ as the desingularization of $X$ via a finite chain of suitable invariant blowing-ups.

Denote with $\E_\R^m$ either $\R^m$ or $\R\PP^m$, let $X\subset\E_\R^m$ be a real algebraic curve and denote $Y:=\cl_{\E_\C^m}^{\zar}(X)$. Let $(\widetilde{Y}\subset\E_\C^k,\Pi)$ be an invariant normalization of $Y$. We claim:
\begin{itemize}
\item[($\bullet$)] If $\widetilde{Z}:=\widetilde{Y}\cap\E_\R^k$ and $Z:=\cl_{\E_\R^m}^{\zar}(X)$, then $\Pi(\widetilde{Z})=Z_{(1)}$.
\end{itemize}
\begin{proof}
Pick a non-singular point $z\in Z$. Then there exists a unique point $w\in\widetilde{Y}$ such that $\Pi(w)=z$. As $\Pi$ is invariant, $\Pi(\sigma_k(w))=\sigma_m(\Pi(w))=\sigma_m(z)=z$, so $w=\sigma_k(w)\in\widetilde{Z}$ (because the fiber of $z$ is a singleton). Thus, $Z_{(1)}\setminus\Sing(Z)\subset\Pi(\widetilde{Z})$. As $\Pi|_{\widetilde{Z}}$ is proper and $Z_{(1)}\setminus\Sing(Z)$ is dense in $Z_{(1)}$, we have $Z_{(1)}\subset\Pi(\widetilde{Z})$. As $\widetilde{Y}$ is an invariant non-singular projective algebraic curve, the intersection $\widetilde{Z}:=\widetilde{Y}\cap\E_\R^k$ is a non-singular projective real algebraic curve, so it is pure dimensional of dimension $1$. As $\Pi^{-1}(\Sing(Y))$ is a finite set and $\Pi|_{\widetilde{Y}\setminus\Pi^{-1}(\Sing(Y))}:\widetilde{Y}\setminus\Pi^{-1}(\Sing(Y))\to Y\setminus\Sing(Y)$ is a biregular isomorphism, we deduce that $\Pi(\widetilde{Z}\setminus\Pi^{-1}(\Sing(Y)))\subset Z_{(1)}$ (because $\widetilde{Z}$ is pure dimensional of dimension $1$ and $\Pi^{-1}(\Sing(Y))$ is a finite set). As $Z_{(1)}$ is a closed semialgebraic set, $\widetilde{Z}\setminus\Pi^{-1}(\Sing(Y))$ is dense in $\widetilde{Z}$ and $\Pi|_{\widetilde{Z}}$ is continuous, we deduce $\Pi(\widetilde{Z})\subset Z_{(1)}$, so $\Pi(\widetilde{Z})=Z_{(1)}$, as required.
\end{proof}

We recall the following result concerning normalizations of invariant rational curves.

\begin{cor}\label{cp1}
Let $X\subset\C\PP^m$ be an invariant rational curve. Let $\Pi:\widetilde{X}\to X$ be an invariant normalization of $X$, where $\widetilde{X}\subset\C\PP^n$ is an invariant non-singular algebraic curve. Then $\C\PP^1$ and $\widetilde{X}$ are biregularly diffeomorphic, so we may assume $\widetilde{X}=\C\PP^1$. 
\end{cor}
\begin{proof}
As $X$ is an invariant rational curve, $X$ is the image of $\C\PP^1$ under an invariant birational (regular) map $\varphi:\C\PP^1\to X$. Thus, there exists a birational map $\psi:=(\Pi|_{\Pi^{-1}(X\setminus\Sing(X))})^{-1}\circ\varphi|_{\varphi^{-1}(X\setminus\Sing(X))}$ between $\C\PP^1$ and $\widetilde{X}$. As both $\C\PP^1$ and $\widetilde{X}$ are non-singular, we deduce by Lemma \ref{curve} that $\psi$ extends to $\C\PP^1$ as a regular map and $\psi^{-1}$ extends to $\widetilde{X}$ as a regular map too. Consequently, $\C\PP^1$ and $\widetilde{X}$ are biregularly diffeomorphic, so we may assume $\widetilde{X}=\C\PP^1$, as required.
\end{proof}

The following two results borrowed from \cite{fe1} (without proof) are crucial to prove the Main results stated in the Introduction.

\begin{lem}[{\cite[Lem.2.2]{fe1}}]\label{fact1}
Let $f:\R\to\R^m$ be a non-constant rational map and $S:=f(\R)$. Then 
\begin{itemize}
\item[(i)] $f$ can be (uniquely) extended to an invariant regular map $F:\C\PP^1\to\C\PP^m$ such that $F(\C\PP^1)=\cl_{\C\PP^m}^{\zar}(S)$.
\item[(ii)] $\cl_{\C\PP^m}^{\zar}(S)$ is an invariant rational curve and if $(\C\PP^1,\Pi)$ is an invariant normalization of $\cl_{\C\PP^m}^{\zar}(S)$, there exists an invariant surjective regular map $\widetilde{F}:\C\PP^1\to\C\PP^1$ such that $F=\Pi\circ\widetilde{F}$.
\item[(iii)] If $f$ is polynomial, we may choose $\Pi$ and $\widetilde{F}$ such that $\pi:=\Pi|_{\R}$ and $\widetilde{f}:=\widetilde{F}|_{\R}$ are polynomial. In particular, $\cl_{\C\PP^m}^{\zar}(S)\cap\mathsf{H}^m_\infty(\C)$ is a singleton $p$ and the analytic set germ $\cl_{\C\PP^m}^{\zar}(S)_p$ is irreducible.
\end{itemize}
\end{lem}

\begin{lem}[{\cite[Lem.2.3]{fe1}}]\label{fact2}
Let $f:=(f_1,\ldots,f_m):\R^n\to\R^m$ be a non-constant rational map such that its image $f(\R^n)$ has dimension $1$. Then
\begin{itemize}
\item[(i)] $f$ factors through $\R$, that is, there exist a rational function $g\in\R(\x)$ and a rational map $h:\R\to\R^m$ such that $f=h\circ g$.
\item[(ii)] If $f$ is in addition a polynomial map, we may also assume that $g$ and $h$ are polynomial.
\end{itemize}
\end{lem}

\subsection{Branches at infinity of a real algebraic curve}

We prove next two announced results in the Introduction (Remarks \ref{infty0}). Although they are surely well-known, we have not found any explicit reference to them in the literature, so we provide explicit proofs of them here. For $\K=\R$ or $\C$ we denote the ring of convergent power series in one variable and coefficients in $\K$ with $\K\{\t\}$ and its field of fractions with $\K(\{\t\})$.

\begin{lem}\label{infty1}
Let $S\subset\R^m$ be a semialgebraic set of dimension $1$. If $\cl_{\C\PP^m}^{\zar}(S)\cap\mathsf{H}^m_\infty(\C)=\{p\}$ and $\cl_{\C\PP^m}^{\zar}(S)_p$ is irreducible, then $\cl_{\C\PP^m}^{\zar}(S)\cap\R^m$ is unbounded.
\end{lem}
\begin{proof}
Consider the complex conjugation $\sigma_m:\C\PP^m\to\C\PP^m$. Both the Zariski closure $\cl_{\C\PP^m}^{\zar}(S)$ and $\mathsf{H}^m_\infty(\C)$ are invariant (under the complex conjugation $\sigma_m$). Observe that $\sigma_m(p)\in\cl_{\C\PP^m}^{\zar}(S)\cap\mathsf{H}^m_\infty(\C)=\{p\}$, so $p=\sigma_m(p)$ and $p\in\mathsf{H}^m_\infty(\R)$. After an invariant projective change of coordinates that keeps invariant the hyperplane at infinity $\mathsf{H}^m_\infty(\C)$, we may assume that $p=[0:1:0:\cdots:0]$. Consider the chart $\{\x_1\neq0\}$ of $\C\PP^m$ and identify it with $\C^m$, so $p=(0,\ldots,0)$. As $\cl_{\C\PP^m}^{\zar}(S)_p$ is irreducible, there exist, after a linear change of coordinates in $\C^m$, by R\"uckert's parameterization \cite[Prop.3.4]{rz}: 
\begin{itemize}
\item irreducible monic polynomials $P_j\in\C\{\x_0\}[\x_j]$ of degree $d_j\geq1$ such that $P_j(0,\x_j)=\x_j^{d_j}$ for $j=2,\ldots,m$, 
\item polynomials $Q_j\in\C\{\x_0\}[\x_2]$ of degree $<d_2$ for $j=3,\ldots,m$,
\item an open (small enough) neighborhood $U\subset\C^m$ of the origin,
\end{itemize}
such that
\begin{multline*}
\cl_{\C\PP^m}^{\zar}(S)\cap U=\{(x_0,x_2,\ldots,x_m)\in U:\ P_2(x_0,x_2)=0, x_j=\tfrac{Q_j(x_0,x_2)}{\Delta_2(x_0)}\text{ for $j=3,\ldots,m$}\}\\
\subset\{(x_0,x_2,\ldots,x_m)\in U:\ P_j(x_0,x_j)=0\text{ for $j=2,\ldots,m$}\}
\end{multline*}
where $\Delta_2\in\C\{\x_0\}\setminus\{0\}$ is the discriminant of $P_2$. By Newton-Puiseux theorem \cite[Prop.4.5]{rz} there exist $\alpha_2\in\C\{\t\}$ and an integer $\ell\geq1$ such that $\alpha_2(0)=0$ and $P_2(\t^\ell,\alpha_2)=0$. Define $\alpha_j:=\tfrac{Q_k(\t^\ell,\t)}{\Delta_2(\t^\ell)}\in\C(\{\t\})$ for $j=3,\ldots,m$. As $P_j(\t^\ell,\alpha_j(\t))=0$ and $P_j\in\C\{\x_0\}[\x_j]$ is a monic polynomial such that $P_j(0,\x_j)=\x_j^{d_j}$, we have $\alpha_j\in\C\{\t\}$ (because $\C\{\t\}$ is integrally closed in $\C(\{\t\})$, as it is a UFD) and $\alpha_j(0)=0$ for $j=2,\ldots,m$. We may assume that each $\alpha_k$ is defined on a disc $D\subset\C$ centered at the origin. The fibers of $\alpha:=[\t^\ell:1:\alpha_2:\cdots:\alpha_m]:D\to\cl_{\C\PP^m}^{\zar}(S)$ are (complex) analytic subsets of $D$. As $\alpha$ is non-constant, we deduce $\alpha^{-1}(\alpha(z))$ have dimension $0$ for each $z\in D$. We conclude by \cite[Ch.VII.Prop.3, pag.131]{n} that $\cl_{\C\PP^m}^{\zar}(S)_p=\im(\alpha)_p$, because $\cl_{\C\PP^m}^{\zar}(S)_p$ is an irreducible analytic set germ of dimension $1$. 

As $S\subset\R^n$ is a semialgebraic set, $\cl_{\C\PP^m}^{\zar}(S)$ is invariant and there exist finitely many homogeneous polynomials $F_k\in\R[\x_0,\x_1,\ldots,\x_m]$ such that $\cl_{\C\PP^m}^{\zar}(S)=\{F_1=0,\ldots,F_s=0\}$. Write $\alpha_j:=\sum_{q\geq0}a_{jq}\t^q$ where each $a_{jq}\in\C$ and define $\beta_j:=\sum_{q\geq0}\ol{a_{jq}}\t^q$. As 
\begin{multline*}
0=\ol{F_k(\t^\ell,1,\alpha_2(\t),\ldots,\alpha_m(\t))}=F_k\Big(\ol{\t}^\ell,1,\sum_{q\geq0}\ol{a_{2q}}\ol{\t}^q,\ldots,\sum_{q\geq0}\ol{a_{mq}}\ol{\t}^q\Big)\\
=F_k(\ol{\t}^\ell,1,\beta_2(\ol{\t}),\ldots,\beta_m(\ol{\t})),
\end{multline*}
we deduce $\beta:=[\t^\ell:1:\beta_2:\cdots:\beta_m]:D\to\cl_{\C\PP^m}^{\zar}(S)$ and $\beta(0)=p\in\mathsf{H}^m_\infty(\R)$. As $\im(\alpha)_p$ is an irreducible analytic set germ of dimension $1$, also $\im(\beta)_p$ is an irreducible analytic set germ of dimension $1$. As $\cl_{\C\PP^m}^{\zar}(S)_p$ is irreducible, $\im(\alpha)_p=\im(\beta)_p$. 

Consequently, for each $t\in D\setminus\{0\}$, there exists $s\in D\setminus\{0\}$ such that $t^\ell=s^\ell$ and $\alpha(t)=\beta(s)$. In particular, there exists an $\ell^{\rm th}$ root of unity $\zeta(t,s)$ such that $s=\zeta(t,s)t$ and $\alpha(t)=\beta(\zeta(t,s)t)$. As there are only $\ell$ possible values of $\zeta(t,s)$, we deduce taking a sequence in $D$ converging to the origin and the Identity Principle that there exists an $\ell^{\rm th}$ root of unity $\zeta$ that does neither depend on $t$ nor on $s$ such that $\alpha(t)=\beta(\zeta t)$ for $t\in D$, so $\alpha(\t)=\beta(\zeta\t)$. Thus, $a_{jq}=\ol{a_{jq}}\zeta^q$ for each $j=2,\ldots,m$ and each $q\geq0$. Write each non-zero $a_{jq}=\rho_{jq}\theta_{jq}$ where $\rho_{jq}\in\R$ is a positive real number and $\theta_{jq}\in\C$ is a complex number of module $1$. Consequently, $\rho_{jq}\theta_{jq}=\rho_{jq}\ol{\theta_{jq}}\zeta^q$. As $\rho_{jq}\neq0$, we have $\theta_{jq}=\ol{\theta_{jq}}\zeta^q$ and (multiplying the previous equality by $\theta_{jq}$) we deduce $\theta_{jq}^2=\zeta^q$. Let $\xi\in\C$ be such that $\xi^2=\zeta$, so $\theta_{jq}^2=\xi^{2q}=(\xi^q)^2$ for $j=2,\ldots,m$ and each $q\geq0$ such that $a_{jq}\neq0$. Thus, there exist $\veps_{jq}\in\{-1,+1\}$ such that $\theta_{jq}=\veps_{jq}\xi^q$ for each $j=2,\ldots,m$ and each $q\geq0$ such that $a_{jq}\neq0$. If $a_{jq}=0$, we define $\rho_{jq}:=0$ and $\veps_{jq}:=1$, so $a_{jq}=\rho_{jq}\veps_{jq}\xi^q$. We deduce 
$$
\alpha_j\Big(\frac{\t}{\xi}\Big)=\sum_{q\geq0}a_{jq}\frac{1}{\xi^q}\t^q=\sum_{q\geq0}\rho_{jq}\veps_{jq}\xi^q\frac{1}{\xi^q}\t^q=\sum_{q\geq0}\rho_{jq}\veps_{jq}\t^q\in\R\{\t\}
$$ 
for each $j=2,\ldots,m$ and $(\frac{\t}{\xi})^\ell=\veps\t^\ell$ for some $\veps\in\{-1,1\}$, because $(\xi^{\ell})^2=\zeta^\ell=1$. Consequently, there exists $\delta>0$ such that $\gamma:=\alpha(\frac{\t}{\xi}):[-\delta,\delta]\to\cl_{\C\PP^m}^{\zar}(S)\cap\R\PP^m$ and $\gamma(0)=p\in\mathsf{H}^m_\infty(\R)$. As $\cl_{\C\PP^m}^{\zar}(S)\cap\mathsf{H}^m_\infty(\R)=\{p\}$, we conclude $\im(\gamma)\setminus\{p\}\subset\cl_{\C\PP^m}^{\zar}(S)\cap\R^m$, so $\cl_{\C\PP^m}^{\zar}(S)\cap\R^m$ is unbounded, as required.
\end{proof}

\begin{lem}\label{infty2}
Let $S\subset\R^m$ be a semialgebraic set of dimension $1$. If $\cl_{\C\PP^m}^{\zar}(S)\cap\mathsf{H}^m_\infty(\C)=\{p\}$ and the analytic set germ $\cl_{\C\PP^m}^{\zar}(S)_p$ has exactly two irreducible components that are conjugated, then $\cl_{\C\PP^m}^{\zar}(S)\cap\R^m$ is bounded.
\end{lem}
\begin{proof}
As $S\subset\R^n$ is a semialgebraic set, $\cl_{\C\PP^m}^{\zar}(S)$ is invariant and there exist finitely many homogeneous polynomials $F_1,\ldots,F_s\in\R[\x_0,\x_1,\ldots,\x_n]$ such that $\cl_{\C\PP^m}^{\zar}(S)=\{F_1=0,\ldots,F_s=0\}$. As $\cl_{\C\PP^m}^{\zar}(S)\cap\mathsf{H}^m_\infty(\C)=\{p\}$, we deduce $\ol{F_k(p)}=F_k(\ol{p})=0$ for $k=1,\ldots,s$. Thus, $\ol{p}\in\cl_{\C\PP^m}^{\zar}(S)\cap\mathsf{H}^m_\infty(\C)=\{p\}$, so $p=\ol{p}\in\mathsf{H}^m_\infty(\R)$. If $T:=\cl_{\C\PP^m}^{\zar}(S)\cap\R^m$ is unbounded, $p\in\cl_{\R\PP^m}(T)$. We embed $\R\PP^m$ as a real algebraic subset of $\R^{(m+1)^2}$, see \cite[\S3.4.2]{bcr}. By the Nash curve selection lemma \cite[Prop.8.1.13]{bcr} there exists a Nash map $\alpha:[-1,1]\to\R\PP^m$ such that $\alpha(0)=p$ and $\alpha((0,1))\subset T$. As $\cl_{\R\PP^m}^{\zar}(S)$ is a real algebraic set and $\alpha((0,1))\subset\cl_{\R\PP^m}^{\zar}(S)$, we deduce by the Identity Principle that $\alpha([-1,1])\subset\cl_{\R\PP^m}^{\zar}(S)$. 

Let $D\subset\C$ be a disc centered in the origin such that there exists a holomorphic extension $\beta:D\to\cl_{\C\PP^m}^{\zar}(S)$ of $\alpha|_{(-\veps,\veps)}$ for some $0<\veps<1$. As the components of $\alpha$ are real analytic functions, the coefficients of the Taylor expansions at the origin of the components of $\beta$ are real numbers. The fibers of $\beta$ are (complex) analytic subsets of $D$. As $\beta$ is non-constant, we deduce $\beta^{-1}(\beta(z))$ have dimension $0$ for each $z\in D$. By \cite[Ch.VII. Prop.3, pag. 131]{n} the set germ $\im(\beta)_p\subset\cl_{\C\PP^m}^{\zar}(S)_p$ is a (complex) analytic set germ of dimension $1$. In fact, it is irreducible, because $D$ is an irreducible (complex) analytic set. As $\cl_{\C\PP^m}^{\zar}(S)_p$ has dimension $1$, $\im(\beta)_p$ is one of the irreducible components of the (complex) analytic set germ $\cl_{\C\PP^m}^{\zar}(S)_p$. As the coefficients of the Taylor expansions at the origin of the components of $\beta$ are real numbers, $\im(\beta)_p$ is invariant under conjugation, which is a contradiction, because $\cl_{\C\PP^m}^{\zar}(S)_p$ has exactly two irreducible components that are conjugated. Consequently, $\cl_{\C\PP^m}^{\zar}(S)\cap\R^m$ is bounded, as required.
\end{proof}

\section{Proofs of the main results}\label{s3}

The main purpose of this section is to prove Theorems \ref{rb1s1}, \ref{pb1} and \ref{ps1}. We also prove Corollary \ref{ps1c} and Proposition \ref{ps2} at the end of the section.

\subsection{Proof of Theorem \ref{rb1s1}}
The chain of implications (i) $\Longrightarrow$ (v) $\Longrightarrow$ (vii) $\Longrightarrow$ (iii) $\Longrightarrow$ (i) $\Longrightarrow$ (ii) $\Longrightarrow$ (vi) $\Longrightarrow$ (viii) $\Longrightarrow$ (iv) follows from the existence of regular surjective maps between $\ol{\Bb}_n$ and $\sph^n\subset\R^{n+1}$ for each $n\geq1$ and viceversa and the fact that a regulous map on a non-singular algebraic curve is a regular map.

We prove next (iv) $\Longrightarrow$ (ix). By \cite[(3.1)(iv)]{fg3} and \cite[Thm.3.11]{fhmm} we deduce $S$ is irreducible and as $f$ is continuous and $\ol{\Bb}_n$ is compact, also $S$ is compact. Now let $f:\R^n\dasharrow\R^m$ be a rational map such that $f$ extends continuously to $\ol{\Bb}_n$ and $f(\ol{\Bb}_n)=S$. By Lemma \ref{fact2} there exist a rational function $g\in\R(\x)$ and a rational map $h:=(\frac{h_1}{h_0},\ldots,\frac{h_m}{h_0}):\R\to\R^m$ such that $f=h\circ g$. By Lemma \ref{fact1} we deduce $\cl_{\R\PP^m}^{\zar}(S)$ is a rational curve.

Let us prove (ix) $\Longrightarrow$ (i). Let $\Pi:\C\PP^1\to\cl_{\C\PP^m}^{\zar}(S)$ be the normalization of $\cl_{\C\PP^m}^{\zar}(S)$. By \S\ref{nac}($\bullet$) we have $\Pi(\R\PP^1)=\cl_{\R\PP^m}^{\zar}(S)_{(1)}$. If $S=\cl_{\R\PP^m}^{\zar}(S)_{(1)}$, then $S$ is by Example \ref{s1rp1} a regular image of $\sph^1$ and consequently a regular image of $\ol{\Bb}_1$. On the other hand, if $S\neq\cl_{\R\PP^m}^{\zar}(S)_{(1)}$, we may assume (after a projective change of coordinates in $\R\PP^1$) that the image of the point at infinite ${[0:1]}$ of $\R\PP^1$ under $\Pi$ belongs to $\cl_{\R\PP^m}^{\zar}(S)_{(1)}\setminus S$. By \cite[Cor.3.5]{fg3} there exists an interval $I\subset\R=\R\PP^1\setminus\{{[0:1]}\}$ that is the one dimensional part of $\Pi^{-1}(S)\cap\R\PP^1$. As $S$ is compact and $\Pi|_{\R\PP^1}$ is proper, the interval $I$ is compact, so we may assume $I=[-1,1]=\ol{\Bb}_1$. Thus, $S$ is a regular image of $\ol{\Bb}_1$, as required.
\qed

\subsection{Proof of Theorem \ref{pb1}}
We first prove the equivalence (i) $\iff$ (ii). The implication left to right is clear. Let us prove the converse. Let $f:=(f_1,\ldots,f_m):\R^n\to\R^m$ be a polynomial map such that $f(\ol{\Bb}_n)=S$. As $S$ is one dimensional, its Zariski closure $Z$ in $\R^m$ is also one dimensional. As the interior in $\R^n$ of $\ol{\Bb}_n$ is non-empty, we deduce by the Identity Principle $f(\R^n)\subset Z$. By Lemma \ref{fact2} there exist a polynomial function $g\in\R[\x]$ and a polynomial map $h:\R\to\R^m$ such that $f=h\circ g$. Consequently, $f(\ol{\Bb}_n)=h(g(\ol{\Bb}_n))$. As $\ol{\Bb}_n$ is compact and connected, $g(\ol{\Bb}_n)$ is a compact (non-trivial) interval $I$ of $\R$ and, after a change of coordinates in $\R$, we may assume $I=[-1,1]$. Thus, $S$ is a polynomial image of $\ol{\Bb}_1$.

Let us prove (i) $\Longrightarrow$ (iii) Let $f:\R\to\R^m$ be a polynomial map such that $f([-1,1])=S$ and define $T:=f(\R)$. As $S$ is one dimensional, its Zariski closure $Z$ in $\R^n$ is also one dimensional. As the interior in $\R$ of $[-1,1]$ is non-empty, we deduce by the Identity Principle $T=f(\R)\subset Z$, so $T$ is also one dimensional. By Theorem \ref{dim1p} we have that $T$ is irreducible, unbounded and $\cl_{\C\PP^m}^{\zar}(S)=\cl_{\C\PP^m}^{\zar}(Z)=\cl_{\C\PP^m}^{\zar}(T)$ is an invariant rational curve such that the set of points at infinity $\cl_{\C\PP^m}^{\zar}(S)\cap\mathsf{H}^m_\infty(\C)$ is a singleton $\{p\}$ and the analytic set germ $\cl_{\C\PP^m}^{\zar}(S)_p$ is irreducible. In addition, $S$ is by \cite[(3.1)(iv)]{fg3} irreducible and compact, because $[-1,1]$ is irreducible and compact.

We prove next (iii) $\Longrightarrow$ (i) As the set of points at infinity $\cl_{\C\PP^m}^{\zar}(S)\cap\mathsf{H}^m_\infty(\C)$ is a singleton $\{p\}$ and the analytic set germ $\cl_{\C\PP^m}^{\zar}(S)_p$ is irreducible, the intersection $\cl_{\C\PP^m}^{\zar}(S)\cap\R^m$ is by Lemma \ref{infty1} unbounded. Thus, the one dimensional component $T$ of $\cl_{\C\PP^m}^{\zar}(S)\cap\R^m$ is unbounded. In addition, $T$ is irreducible and $\cl_{\C\PP^m}^{\zar}(T)=\cl_{\C\PP^m}^{\zar}(S)$ is an invariant rational curve such that the set of points at infinity $\cl_{\C\PP^m}^{\zar}(T)\cap\mathsf{H}^m_\infty(\C)$ is a singleton $\{p\}$ and the analytic set germ $\cl_{\C\PP^m}^{\zar}(T)_p$ is irreducible. Let $\Pi:=[\Pi_0:\cdots:\Pi_m]:\C\PP^1\to\cl_{\C\PP^m}^{\zar}(T)$ be an invariant normalization of the rational curve $\cl_{\C\PP^m}^{\zar}(T)$ such that $\Pi({[0:1]})=p\in\mathsf{H}^m_\infty(\C)$. We claim: {\em $\Pi|_\R$ is a polynomial map}.

As $\Pi$ is invariant, we may assume that the components $\Pi_k$ of $\Pi$ are real homogeneous polynomials of certain common degree $d$. Thus, $\Pi|_{\R\PP^1}:\R\PP^1\to\cl_{\R\PP^m}^{\zar}(T)$ is a real polynomial map. As $\cl_{\C\PP^m}^{\zar}(T)\cap\mathsf{H}^m_\infty(\C)=\{p\}$ and the analytic set germ $\cl_{\C\PP^m}^{\zar}(T)_p$ is irreducible, 
$$
\{\Pi_0=0\}=\Pi^{-1}(\cl_{\C\PP^m}^{\zar}(S)\cap\mathsf{H}^m_\infty(\C))=\Pi^{-1}(\{p\})
$$ 
is a singleton (see \S\ref{nac}). As $\Pi({[0:1]})=p$, we deduce $\{\Pi_0=0\}=\{{[0:1]}\}$. Thus, $\Pi_0=\lambda\t_0^d$ for some $\lambda\in\R\setminus\{0\}$ and, in fact, we may assume $\lambda=1$. Consequently, $\Pi|_\R$ is a polynomial map (recall that $\R\equiv\{\t_0=1\}$), as claimed.

As $S$ is irreducible, there exists by \cite[Thm.3.15]{fg3} a $1$-dimensional connected component $I$ of $\Pi^{-1}(S)\cap\R\PP^1$ such that $\Pi(I)=S$. As $\Pi({[0:1]})=p$, we deduce $I\subset\R$. As $\Pi$ is proper and $S$ is compact, $I\subset\R$ is compact. We conclude $I$ is a non-trivial compact interval. Thus, after a change of coordinates in $\R$, we may assume $I:=[-1,1]=\ol{\Bb}_1$, so $S=\Pi|_{\R}(\ol{\Bb}_1)$ is a polynomial image of $\ol{\Bb}_1$, as required.
\qed

\subsection{Polynomial maps on the circle and complex Laurent polynomials}\label{pclp}
Before proving Theorem \ref{ps1} we recall that the restriction to $\sph^1\subset\R^2$ of a polynomial map $g:=(g_1,g_2):\R^2\to\R^2$ coincides with the restriction to $\sph^1\subset\C\equiv\R^2$ of a Laurent polynomial $\Gamma\in\C[\z,\z^{-1}]$. Namely, if $g_1:=\sum_{k,\ell}a_{k\ell}\x^k\y^\ell$ and $g_2:=\sum_{k,\ell}b_{k\ell}\x^k\y^\ell$, then 
$$
\Gamma:=\sum_{k,\ell}(a_{k\ell}+{\tt i}b_{k\ell})\Big(\frac{\z+\ol{\z}}{2}\Big)^k\Big(\frac{\z-\ol{\z}}{2{\tt i}}\Big)^\ell=\sum_{k,\ell}(a_{k\ell}+{\tt i}b_{k\ell})\Big(\frac{\z}{2}+\frac{1}{2\z}\Big)^k\Big(\frac{\z}{2{\tt i}}-\frac{1}{2{\tt i}\z}\Big)^\ell\in\C[\z,\z^{-1}],
$$
where $\z=\x+{\tt i}\y$ and $\z\ol{\z}=1$. Conversely, if $\Gamma=\sum_{k=-m}^n\alpha_k\z^k\in\C[\z,\z^{-1}]$ is a Laurent polynomial for some integers $m,n\geq0$ and $\alpha_k:=a_k+{\tt i}b_k$ where $a_k,b_k\in\R$ for each $k$, the restriction $\Gamma|_{\sph^1}:\sph^1\to\C$ equals
$$
\sum_{k=-m}^n\alpha_k\z^k=\sum_{k=0}^n\alpha_k\z^k+\sum_{k=0}^m\alpha_{-k}\ol{\z}^k=\sum_{k=0}^n(a_k+{\tt i}b_k)(\x+{\tt i}\y)^k+\sum_{k=0}^m(a_{-k}+{\tt i}b_{-k})(\x-{\tt i}\y)^k.
$$
Considering the real and imaginary parts of the previous expression, we find $g_1,g_2\in\R[\x,\y]$ such that the polynomial map $g:=(g_1,g_2):\R^2\to\R^2$ satisfies $g|_{\sph^1}=\Gamma|_{\sph^1}$ after identifying $\C\equiv\R^2$.

\subsection{Proof of Theorem \ref{ps1}}
(i) $\Longrightarrow$ (ii) Suppose first $S$ is a polynomial image of $\sph^1$. As $\sph^1$ is compact, {\em $S$ is a compact semialgebraic set} and by \cite[(3.1)(iv)]{fg3} $S$ is {\em irreducible}.

Let $g:=(g_1,\ldots,g_m):\R^2\to\R^m$ be a polynomial map, where $g_i\in\R[\x_1,\x_2]$, such that $g(\sph^1)=S$. Let $d:=\max\{\deg(g_i):\ i=1,\ldots,m\}$ and define $G_i:=g_i(\frac{\x_1}{\x_0},\frac{\x_2}{\x_0})\x_0^d$ for $i=1,\ldots,m$ and $G_0:=\x_0^d$. Consider the polynomial map $G:=[G_0:G_1:\cdots:G_m]:\C\PP^2\dasharrow\C\PP^m$ and its restriction to $X:=\{\x_1^2+\x_2^2-\x_0^2=0\}\subset\C\PP^2$, which is the Zariski closure of $\sph^1$ in $\C\PP^2$. By Lemma \ref{curve} $G|_X$ extends to $X$ as a (unique) invariant regular map that we denote with $G$. As $G$ is continuous for the Zariski topology, we deduce $G(X)\subset\cl_{\C\PP^m}^{\zar}(S)$. By \cite[Prop.(2.31)]{m1} it contains a non-empty Zariski open subset of $\cl_{\C\PP^m}^{\zar}(S)$. As $G$ is proper and $\cl_{\C\PP^m}^{\zar}(S)$ is irreducible, we conclude by \cite[Thm.(2.33)]{m1} $G(X)=\cl_{\C\PP^m}^{\zar}(S)$.

Consider the parameterization $\Phi:\C\PP^1\to X,\ [\t_0:\t_1]\to[\t_0^2+\t_1^2:2\t_0\t_1:\t_1^2-\t_0^2]$, which is the regular extension to $\C\PP^1$ of the stereographic projection
$$
\varphi:\R\to\sph^1,\ t\mapsto\Big(\frac{2t}{1+t^2},\frac{t^2-1}{1+t^2}\Big)
$$
whose image is $\sph^1\setminus\{(0,1)\}$. Consider the composition $F:=[F_0:\cdots:F_m]=G\circ\Phi:\C\PP^1\to\cl_{\C\PP^m}^{\zar}(S)$, which is surjective, and observe that $F_0=(\t_0^2+\t_1^2)^d$. The $\gcd(F_0,\ldots,F_m)$ is an invariant divisor of $(\t_0^2+\t_1^2)^d$. After dividing each $F_i$ by such greatest common divisor, we assume $F_0=(\t_0^2+\t_1^2)^p$ for some integer $p\geq1$ and $\gcd(F_0,\ldots,F_m)=1$.

Let $(\widetilde{Y}\subset\C\PP^k,\Pi)$ be an invariant normalization of $Y:=\cl_{\C\PP^m}^{\zar}(S)$. The composition $\Pi^{-1}\circ F:\C\PP^1\dashrightarrow \widetilde{Y}$ defines an invariant rational map that can be extended to an invariant surjective regular map $\widetilde{F}:\C\PP^1\to\widetilde{Y}$ such that $F=\Pi\circ\widetilde{F}$. Observe that $\widetilde{Y}$ is by \cite[Cor.(7.6), Cor.(7.20)]{m1} a non-singular curve of arithmetic genus $0$, that is, a non-singular rational curve \cite[Cor.(7.17)]{m1}. Consequently, we may take $\widetilde{Y}=\C\PP^1$ and {\em $\cl_{\C\PP^m}^{\zar}(S)$ is an invariant rational curve}. Thus, $\Pi(\R\PP^1)=\cl_{\R\PP^m}^{\zar}(S)_{(1)}$ (see \S\ref{nac}($\bullet$)).

Write $\Pi:=(\Pi_0,\ldots,\Pi_m)$ and $\widetilde{F}:=(\widetilde{F}_0,\widetilde{F}_1)$ where $\Pi_i,\widetilde{F}_j\in\R[\x_0,\x_1]$ are homogeneous polynomials and $\widetilde{F}_0,\widetilde{F}_1$ are relatively prime. We claim: {\em One of the following situations hold:
\begin{itemize}
\item[(1)] $\widetilde{F}$ is an invariant projective change of coordinates in $\C\PP^1$, so we assume $F=\Pi$ and $\Pi_0=F_0=(\x_0^2+\x_1^2)^p$.
\item[(2)] After an invariant projective change of coordinates in $\C\PP^1$, we have $\Pi_0=\x_0^e$ for some $e\geq1$ and $\widetilde{F}_0=(\x_0^2+\x_1^2)^\ell$ where $\ell\geq1$ and $p=\ell e$.
\item[(3)] After an invariant projective change of coordinates in $\C\PP^1$, we have $\Pi_0=(\x_0^2+\x_1^2)^{e_0}$, $\widetilde{F}_1-{\tt i}\widetilde{F}_0=\lambda_1(\x_0+{\tt i}\x_1)^k$, $\widetilde{F}_1+{\tt i}\widetilde{F}_0=\ol{\lambda_1}(\x_0-{\tt i}\x_1)^k$ for some positive integers $e_0,k$ such that $p=ke_0$ and some $\lambda_1\in\C\setminus\{0\}$.
\end{itemize}
}

Observe first that $\widetilde{F}$ is not constant, because it is surjective. Factorize
$$
\Pi_0=u\prod_{i=1}^e(a_i\x_1-b_i\x_0)\in\C[\x_0,\x_1]
$$
where $u\in\R\setminus\{0\}$, $a_i\in\{0,1\}$, $b_i\in\C$, $b_i=1$ if $a_i=0$ and $(a_i,b_i)\neq(0,0)$ for $i=1,\ldots,m$. Denote ${\tt p}_i:=\widetilde{F}_i(1,\x_1)\in\R[\x_1]$ and observe
\begin{equation}\label{prod}
\prod_{i=1}^e(a_i{\tt p}_1-b_i{\tt p}_0)=\Pi_0({\tt p}_0,{\tt p}_1)=F_0(1,\x_1)=(1+\x_1^2)^p=(1+{\tt i}\x_1)^p(1-{\tt i}\x_1)^p.
\end{equation}
We proceed in several steps:

\noindent{\sc Step 1.} Suppose first $a_1=1$ and $b_1\in\C\setminus\R$. As all the involved rational maps are invariant, we may assume $a_2=1$ and $b_2=\ol{b_1}$. As ${\tt p}_0,{\tt p}_1$ are relatively prime (and at least one of them is non constant), we deduce: 
{\em
$$
\begin{cases}
{\tt p}_1-b_1{\tt p}_0=\lambda_1(1+{\tt i}\x_1)^{k_1},\\
{\tt p}_1-\ol{b_1}{\tt p}_0=\ol{\lambda_1}(1-{\tt i}\x_1)^{k_1}
\end{cases}
$$
for some $k_1\geq1$ and $\lambda_1\in\C\setminus\{0\}$.} 

Otherwise, either $k_1=0$, which is a contradiction, because at least one between ${\tt p}_0,{\tt p}_1$ is not constant, or $(1+\x_1^2)$ divides both ${\tt p}_1-b_1{\tt p}_0$ and ${\tt p}_1-\ol{b_1}{\tt p}_0$, so $(1+\x_1^2)$ divides both ${\tt p}_0$ and ${\tt p}_1$, which is a contradiction, because ${\tt p}_0,{\tt p}_1$ are relatively prime. We have the system:
$$
\begin{cases}
{\tt p}_1-b_1{\tt p}_0=\lambda_1(1+{\tt i}\x_1)^{k_1},\\
{\tt p}_1-\ol{b_1}{\tt p}_0=\ol{\lambda_1}(1-{\tt i}\x_1)^{k_1}.
\end{cases}
$$
Consequently,
\begin{equation}\label{p0p1}
{\tt p}_0=\frac{\lambda_1(1+{\tt i}\x_1)^{k_1}-\ol{\lambda_1}(1-{\tt i}\x_1)^{k_1}}{\ol{b_1}-b_1}\quad\text{and}\quad{\tt p}_1=\frac{\ol{b_1}\lambda_1(1+{\tt i}\x_1)^{k_1}-b_1\ol{\lambda_1}(1-{\tt i}\x_1)^{k_1}}{\ol{b_1}-b_1}.
\end{equation}

\noindent{\sc Step 2.} Suppose there exists a root $[a_3:b_3]\in\C\PP^1$ of $\Pi_0$ different from $[1:b_1]$ and $[1:\ol{b_1}]$. Then $a_3{\tt p}_1-b_3{\tt p}_0=\lambda_3(1+{\tt i}\x_1)^{k_3}(1-{\tt i}\x_1)^{k_3'}$ for some $k_3,k_3'\geq0$ and $\lambda_3\in\C\setminus\{0\}$. Suppose first $k_3>0$ and consider the system:
$$
\begin{cases}
{\tt p}_1-b_1{\tt p}_0=\lambda_1(1+{\tt i}\x_1)^{k_1},\\
a_3{\tt p}_1-b_3{\tt p}_0=\lambda_3(1+{\tt i}\x_1)^{k_3}(1-{\tt i}\x_1)^{k_3'}.
\end{cases}
$$
Then ${\tt p}_0,{\tt p}_1$ share the irreducible factor $1+{\tt i}\x_1$, which is a contradiction. Consequently, $k_3=0$. If $k_3'>0$, we consider the system:
$$
\begin{cases}
{\tt p}_1-\ol{b_1}{\tt p}_0=\ol{\lambda_1}(1-{\tt i}\x_1)^{k_1},\\
a_3{\tt p}_1-b_3{\tt p}_0=\lambda_3(1-{\tt i}\x_1)^{k_3'}
\end{cases}
$$
and we conclude that ${\tt p}_0,{\tt p}_1$ share the irreducible factor $1-{\tt i}\x_1$, which is a contradiction. Consequently, $k_3'=0$ and $a_3{\tt p}_1-b_3{\tt p}_0=\lambda_3$. By \eqref{p0p1}
$$
\lambda_3(\ol{b_1}-b_1)=(a_3{\tt p}_1-b_3{\tt p}_0)(\ol{b_1}-b_1)=(a_3\ol{b_1}\lambda_1-b_3\lambda_1)(1+{\tt i}\x_1)^{k_1}-(a_3b_1\ol{\lambda_1}-b_3\ol{\lambda_1})(1-{\tt i}\x_1)^{k_1}.
$$
We substitute $\x_1=0$, $\x_1={\tt i}$ and $\x_1=-{\tt i}$ and obtain:
$$
\begin{cases}
\lambda_3(\ol{b_1}-b_1)=(a_3\ol{b_1}\lambda_1-b_3\lambda_1)-(a_3b_1\ol{\lambda_1}-b_3\ol{\lambda_1}),\\
\lambda_3(\ol{b_1}-b_1)=-(a_3b_1\ol{\lambda_1}-b_3\ol{\lambda_1})2^{k_1},\\
\lambda_3(\ol{b_1}-b_1)=(a_3\ol{b_1}\lambda_1-b_3\lambda_1)2^{k_1}.
\end{cases}
$$
Thus, using second and third equations, we have $(a_3\ol{b_1}\lambda_1-b_3\lambda_1)=-(a_3b_1\ol{\lambda_1}-b_3\ol{\lambda_1})$. Using now the first equation, we deduce $\lambda_3(\ol{b_1}-b_1)=2(a_3\ol{b_1}\lambda_1-b_3\lambda_1)$. As $\lambda_3(\ol{b_1}-b_1)\neq0$, we have $a_3\ol{b_1}\lambda_1-b_3\lambda_1\neq0$. Thus, 
\begin{multline*}
2(a_3\ol{b_1}\lambda_1-b_3\lambda_1)=(a_3\ol{b_1}\lambda_1-b_3\lambda_1)((1+{\tt i}\x_1)^{k_1}+(1-{\tt i}\x_1)^{k_1})\\ 
\leadsto\ 2=(1+{\tt i}\x_1)^{k_1}+(1-{\tt i}\x_1)^{k_1}.
\end{multline*}
If we substitute $\x_1=2{\tt i}$, we deduce $2=(-1)^{k_1}+3^{k_1}$, which only holds if $k_1=1$. This means by \eqref{p0p1}
\begin{equation}\label{p0p12}
{\tt p}_0=\frac{\lambda_1(1+{\tt i}\x_1)-\ol{\lambda_1}(1-{\tt i}\x_1)}{\ol{b_1}-b_1}\quad\text{and}\quad{\tt p}_1=\frac{\ol{b_1}\lambda_1(1+{\tt i}\x_1)-b_1\ol{\lambda_1}(1-{\tt i}\x_1)}{\ol{b_1}-b_1},
\end{equation}
so $[\widetilde{F}_0:\widetilde{F}_1]:\C\PP^1\to\C\PP^1$ is an invariant projective change of coordinates. We may assume, after changing $\Pi$ by $F$, that $\Pi=F$, which corresponds to situation (1) above. 

\noindent{\sc Step 3}. Assume next that no root $[a_3:b_3]\in\C\PP^1$ of $\Pi_0$ is different from either $[1:b_1]$ or $[1:\ol{b_1}]$. Thus, the degree $e$ of $\Pi_0$ is even, say $e=2e_0$, and
$$
\Pi_0=u(\x_1-b_1\x_0)^{e_0}(\x_1-\ol{b_1}\x_0)^{e_0}.
$$
After an invariant projective change of coordinates in $\C\PP^1$ that maps $[1:b_1]$ to $[1:{\tt i}]$ and $[1:\ol{b_1}]$ to $[1:-{\tt i}]$, we may assume $\Pi_0=u(\x_0^2+\x_1^2)^{e_0}$, 
$\Pi_0({\tt p}_0,{\tt p}_1)=u({\tt p}_0^2+{\tt p}_1^2)^{e_0}=(1+\x_1^2)^p$ and
$$
\begin{cases}
{\tt p}_1-{\tt i}{\tt p}_0=\lambda_1(1+{\tt i}\x_1)^{k_1},\\
{\tt p}_1+{\tt i}{\tt p}_0=\ol{\lambda_1}(1-{\tt i}\x_1)^{k_1}.
\end{cases}
$$
Consequently,
\begin{equation}\label{cp1cp1}
\begin{cases}
\widetilde{F}_1-{\tt i}\widetilde{F}_0=\lambda_1(\x_0+{\tt i}\x_1)^{k_1},\\
\widetilde{F}_1+{\tt i}\widetilde{F}_0=\ol{\lambda_1}(\x_0-{\tt i}\x_1)^{k_1}.
\end{cases}
\end{equation}
As $\Pi_0(\widetilde{F}_0,\widetilde{F}_1)=(\x_0^2+\x_1^2)^p$, we have $k_1e_0=p$. This means that we are in situation (3) above. In Remark \ref{degree} we will use \eqref{cp1cp1} to better understand the regular map $\widetilde{F}$. 

\noindent{\sc Step 4}. Assume next the roots of $\Pi_0$ belong to $\R\PP^1$. Suppose after interchanging the variables $\x_0,\x_1$ if necessary that $a_1=1$ and $[1:b_1]\neq[a_2:b_2]$. As ${\tt p}_i\in\R[\x_1]$ for $i=0,1$, we have by \eqref{prod} the system:
$$
\begin{cases}
{\tt p}_1-b_1{\tt p}_0=\lambda_1(1+\x_1^2)^{k_1},\\
a_2{\tt p}_1-b_2{\tt p}_0=\lambda_2(1+\x_1^2)^{k_2},
\end{cases}
$$
for some $\lambda_1,\lambda_2\in\R\setminus\{0\}$ and $k_1,k_2\geq0$. As ${\tt p}_0,{\tt p}_1$ are relatively prime, we deduce that either $k_1=0$ or $k_2=0$. Suppose there exists $[a_3:b_3]\in\R\PP^1$ different from $[1:b_1]$ and $[a_2:b_2]$. Then there exists $\lambda_3\in\R\setminus\{0\}$ such that 
$$
a_3{\tt p}_1-b_3{\tt p}_0=\lambda_3(1+\x_1^2)^{k_3}
$$
for some $k_3\geq0$. As ${\tt p}_0$ and ${\tt p}_1$ are relatively prime, we deduce that in the triple $k_1,k_2,k_3$ there are two integers, which are zero. This implies solving the corresponding linear system that ${\tt p_0}$ and ${\tt p_1}$ are constant, which is a contradiction. Thus, $\Pi_0$ has only two different roots $[1:b_1]$ and $[a_2:b_2]$ of multiplicities $e_0$ and $e_1$ such that $e_0+e_1=e$. Consequently, after an invariant projective change of coordinates in $\C\PP^1$ that maps $[1:b_1]$ to $[1:0]$ and $[a_2:b_2]$ to ${[0:1]}$, we may assume $\Pi_0=\x_0^{e_0}\x_1^{e_1}$ (where $e_0,e_1\geq0$). We have $(\x_0^2+\x_1^2)^p=F_0=\Pi_0(\widetilde{F}_0,\widetilde{F}_1)=\widetilde{F}_0^{e_0}\widetilde{F}_1^{e_1}$, which is a contradiction, because $\widetilde{F}_0,\widetilde{F}_1\in\R[\x_0,\x_1]$ are relatively prime. Thus, $\Pi_0=u(a_1\x_1-b_1\x_0)^e$ and after an invariant projective change of coordinates in $\C\PP^1$ we may assume $\Pi_0=u\x_0^e$. As $=(\x_0^2+\x_1^2)^p=F_0=\widetilde{F}_0^e$, we conclude that $\widetilde{F}_0=(\x_0^2+\x_1^2)^\ell$ for some $\ell\geq1$ such that $p=\ell e$. This corresponds to situation (2) above.

We deduce $\Pi^{-1}(\cl_{\C\PP^m}^{\zar}(S)\cap\mathsf{H}^m_\infty(\C))$ is either equal to ${[0:1]}$ (in situation (2)) or to $\{[1:{\tt i}],[1:-{\tt i}]\}$ (in situations (1) and (3)). We distinguish several cases:

\noindent{\sc Case 1.} $\Pi^{-1}(\cl_{\C\PP^m}^{\zar}(S)\cap\mathsf{H}^m_\infty(\C))=\{{[0:1]}\}$. Then $\cl_{\C\PP^m}^{\zar}(S)\cap\mathsf{H}^m_\infty(\C)=\{p:=\Pi({[0:1]})\}$, so it is a singleton that belongs to $\mathsf{H}^m_\infty(\R)$. As $\cl_{\C\PP^m}^{\zar}(S)_p=\Pi(\C\PP^1_{[0:1]})$, the analytic set germ $\cl_{\C\PP^m}^{\zar}(S)_p$ is irreducible. Consequently, by Lemma \ref{infty1} $\cl_{\R\PP^m}^{\zar}(S)_{(1)}\cap\R^m$ is unbounded.

\noindent{\sc Case 2.} $\Pi^{-1}(\cl_{\C\PP^m}^{\zar}(S)\cap\mathsf{H}^m_\infty(\C))=\{[1:{\tt i}],[1:-{\tt i}]\}$ and $\Pi([1:{\tt i}])=\Pi([1:-{\tt i}])$. Then $\cl_{\C\PP^m}^{\zar}(S)\cap\mathsf{H}^m_\infty(\C)=\{p:=\Pi({[1:{\tt i}]})\}$, so it is a singleton that belongs to $\mathsf{H}^m_\infty(\R)$, because $p=\Pi({[1:{\tt i}]})=\Pi(\sigma_1({[1:{\tt i}]}))=\sigma_m(\Pi({[1:{\tt i}]}))=\sigma_m(p)$. Observe that $\cl_{\C\PP^m}^{\zar}(S)_p=\Pi(\C\PP^1_{[1:{\tt i}]})\cup\Pi(\C\PP^1_{[1:-{\tt i}]})$, so the analytic set germ $\cl_{\C\PP^m}^{\zar}(S)_p$ has exactly two irreducible components that are conjugated, that is, $\sigma_m(\Pi(\C\PP^1_{[1:{\tt i}]}))=\Pi(\sigma_1(\C\PP^1_{[1:{\tt i}]}))=\Pi(\C\PP^1_{[1:-{\tt i}]})$, because $\Pi$ is invariant. In this case $\Pi(\R\PP^1)=\cl_{\R\PP^m}^{\zar}(S)_{(1)}\subset\R^m$ (see \S\ref{nac}($\bullet$)) is compact by Lemma \ref{infty2}.

\noindent{\sc Case 3.} $\Pi^{-1}(\cl_{\C\PP^m}^{\zar}(S)\cap\mathsf{H}^m_\infty(\C))=\{[1:{\tt i}],[1:-{\tt i}]\}$ and $q:=\Pi([1:{\tt i}])\neq\Pi([1:-{\tt i}])=\ol{q}$. Then $\cl_{\C\PP^m}^{\zar}(S)\cap\mathsf{H}^m_\infty(\C)=\{q,\ol{q}\}$ and $q,\ol{q}\not\in\mathsf{H}^m_\infty(\R)$. Observe that $\cl_{\C\PP^m}^{\zar}(S)_q=\Pi(\C\PP^1_{[1:{\tt i}]})$ and $\cl_{\C\PP^m}^{\zar}(S)_{\ol{q}}=\Pi(\C\PP^1_{[1:-{\tt i}]})$ are irreducible and conjugated, because $\Pi$ is invariant, so 
$$
\sigma_m(\cl_{\C\PP^m}^{\zar}(S)_q)=\sigma_m(\Pi(\C\PP^1_{[1:{\tt i}]}))=\Pi(\sigma_1(\C\PP^1_{[1:{\tt i}]})))=\Pi(\C\PP^1_{[1:-{\tt i}]})=\cl_{\C\PP^m}^{\zar}(S)_{\ol{q}}.
$$
In this case $\Pi(\R\PP^1)=\cl_{\R\PP^m}^{\zar}(S)_{(1)}\subset\R^m$ (see \S\ref{nac}($\bullet$)) is compact, because $\cl_{\R\PP^m}^{\zar}(S)\cap\mathsf{H}^m_\infty(\R)=\varnothing$.

It remains to show for {\sc Cases} 2 and 3 that $S=\cl_{\R\PP^m}^{\zar}(S)_{(1)}$. We have already proved in both cases that $\Pi(\R\PP^1)=\cl_{\R\PP^m}^{\zar}(S)_{(1)}\subset\R^m$ is compact, so $\cl_{\C\PP^m}^{\zar}(S)\cap\R^m=\cl_{\R\PP^m}^{\zar}(S)_{(1)}\cup(\cl_{\R\PP^m}^{\zar}(S)_{(0)}\cap\R^m)$ is a compact set. As $S=g(\sph^1)\subset\cl_{\R\PP^m}^{\zar}(S)\cap\R^m$ is a pure dimensional semialgebraic set of dimension $1$, we deduce $S\subset\cl_{\R\PP^m}^{\zar}(S)_{(1)}$.

Denote $Z:=\cl_{\C\PP^m}^{\zar}(S)\cap\C^m$, which is an irreducible algebraic set of $\C^m$ of dimension $1$, and let $\rho:\C^m\to\C^2$ be an invariant generic projection such that $Z':=\rho(Z)\subset\C^2$ is an algebraic curve and the restriction $\rho|_{Z}:Z\to Z'$ is (surjective and) generically $1$-$1$. To construct such a projection use Finiteness of Noether's normalization \cite[Thm.1.5.19]{jp} and algebraicity of generic projections \cite[Lem.2.1.6, Thm.2.2.8]{jp}. As $\rho|_Z$ is invariant and generically $1$-$1$, only finitely many points of $Z\setminus\R^m$ are mapped onto $Z'\cap\R^2$ (because conjugated points have conjugated images). Thus, as $Z\cap\R^m=\cl_{\C\PP^m}^{\zar}(S)\cap\R^m$ is a compact set and $(Z'\cap\R^2)_{(1)}\setminus\rho(Z\setminus\R^m)$ is dense in the pure dimensional semialgebraic set $(Z'\cap\R^2)_{(1)}$, we deduce $(Z'\cap\R^2)_{(1)}\subset\rho(Z\cap\R^m)\subset Z'\cap\R^2$. Consequently, $(Z'\cap\R^2)_{(1)}$ is a compact set, so $Z'\cap\R^2$ is a compact (real) algebraic set. 

Consider the polynomial map $\rho\circ g:\R^2\to\R^2$ and let $\Gamma\in\C[\z,\z^{-1}]$ be a Laurent polynomial such that $\Gamma(\sph^1)=(\rho\circ g)(\sph^1)$ after identifying $\C$ with $\R^2$ (see \S\ref{pclp}). We have $\Gamma(\sph^1)=(\rho\circ g)(\sph^1)\subset\rho(Z\cap\R^m)\subset Z'\cap\R^2$, which is a compact set. By \cite[Thm.2.1]{ky} the difference $(Z'\cap\R^2)\setminus (\rho\circ g)(\sph^1)$ is a finite set (maybe empty), so $\rho(Z\cap\R^m)\setminus\rho(g(\sph^1))$ is also a finite set (maybe empty). As $\rho|_Z$ is generically $1$-$1$, we deduce that $(Z\cap\R^m)\setminus g(\sph^1)$ is a finite set, so $(Z\cap\R^m)_{(1)}\setminus g(\sph^1)$ is a finite set. As $\sph^1$ is compact and $(Z\cap\R^m)_{(1)}$ is pure dimensional of dimension $1$, we conclude: $S=g(\sph^1)=(Z\cap\R^m)_{(1)}=\cl_{\R\PP^m}^{\zar}(S)_{(1)}$.

(ii) $\Longrightarrow$ (i) As $\cl_{\C\PP^m}^{\zar}(S)$ is an invariant rational curve, there exists by Corollary \ref{cp1} an invariant normalization $\Pi:=[\Pi_0:\cdots:\Pi_m]:\C\PP^1\to\cl_{\C\PP^m}^{\zar}(S)$, which is a surjective regular map. In particular, $\Pi(\R\PP^1)=\cl_{\R\PP^m}^{\zar}(S)_{(1)}$ (see \S\ref{nac}($\bullet$)). We distinguish three cases:

\noindent{\sc Case 1.} $\cl_{\C\PP^m}^{\zar}(S)\cap\mathsf{H}^m_\infty(\C)=\{p\}$ is a singleton and the analytic set germ $\cl_{\C\PP^m}^{\zar}(S)_p$ is irreducible. Thus, we may assume
$$
\{\Pi_0=0\}=\Pi^{-1}(\cl_{\C\PP^m}^{\zar}(S)\cap\mathsf{H}^m_\infty(\C))=\{{[0:1]}\}.
$$
Consequently, $\Pi_0=\lambda\t_0^d$ for some $d\geq1$ and some $\lambda\in\R\setminus\{0\}$ (recall that $\Pi$ is invariant), so we may assume $\lambda=1$. This means that the restriction 
$$
\Pi|_{\R}:\R\equiv\R\PP^1\setminus\{{[0:1]}\}\to\cl_{\C\PP^m}^{\zar}(S)\cap\R^m
$$ 
is a polynomial map. As $S$ is irreducible and $1$-dimensional, $S\subset\cl_{\R\PP^m}^{\zar}(S)_{(1)}\setminus\mathsf{H}^m_\infty(\R)=\Pi(\R)$ (the last equality holds, because $\Pi^{-1}(\cl_{\R\PP^m}^{\zar}(S)_{(1)}\cap\mathsf{H}^m_\infty(\R))=\{{[0:1]}\}$). As $S$ is irreducible, the one dimensional component $I$ of $\Pi^{-1}(S)$ is by \cite[Thm.3.15]{fg3} connected, $I\subset\R$ (because $\Pi^{-1}(\cl_{\C\PP^m}^{\zar}(S)\cap\mathsf{H}^m_\infty(\C))=\{{[0:1]}\}$) and $\Pi(I)=S$. As $S$ is compact and $\Pi$ is proper, also $I$ is compact, so $I$ is a compact interval. After an affine change of coordinates, we may assume $I=[-1,1]$. As $I$ is a polynomial image of $\sph^1$, the same happens to $S$.

\noindent{\sc Case 2.} $\cl_{\C\PP^m}^{\zar}(S)\cap\mathsf{H}^m_\infty(\C)=\{p\}$ is a singleton that belongs to $\mathsf{H}^m_\infty(\R)$, the analytic set germ $\cl_{\C\PP^m}^{\zar}(S)_p$ has exactly two irreducible components that are conjugated and $S=\cl_{\R\PP^m}^{\zar}(S)_{(1)}$.

\noindent{\sc Case 3.} $\cl_{\C\PP^m}^{\zar}(S)\cap\mathsf{H}^m_\infty(\C)=\{q,\ol{q}\}$ (where $q,\ol{q}\not\in\mathsf{H}^m_\infty(\R)$), both germs $\cl_{\C\PP^m}^{\zar}(S)_q$ and $\cl_{\C\PP^m}^{\zar}(S)_{\ol{q}}$ are irreducible and conjugated and $S=\cl_{\R\PP^m}^{\zar}(S)_{(1)}$. 

We prove both {\sc Cases 2} and 3 simultaneously. Recall that $\Pi:=[\Pi_0:\cdots:\Pi_m]:\C\PP^1\to\cl_{\C\PP^m}^{\zar}(S)$ is an invariant normalization of $\cl_{\C\PP^m}^{\zar}(S)$. Both in {\sc Cases 2} and 3 we may assume
$$
\{\Pi_0=0\}=\Pi^{-1}(\cl_{\C\PP^m}^{\zar}(S)\cap\mathsf{H}^m_\infty(\C))=\{[1:{\tt i}],[1:-{\tt i}]\}.
$$
As $\Pi$ is invariant, we deduce $\Pi_0:=\lambda(\x_0+{\tt i}\x_1)^p(\x_0-{\tt i}\x_1)^p=\lambda(\x_0^2+\x_1^2)^p$ for some integer $p\geq1$ and some $\lambda\in\R\setminus\{0\}$, so we may assume $\lambda=1$. As all the components of $\Pi$ are homogeneous polynomials of the same degree, we deduce that such degree is $2p$. In addition, by \S\ref{nac}($\bullet$) we have $\Pi(\R\PP^1)=\cl_{\R\PP^m}^{\zar}(S)_{(1)}=S$. 

Consider the regular map
$$
\Psi:\{\x_1^2+\x_2^2-\x_0^2=0\}\to\C\PP^1,\ [x_0:x_1:x_2]\mapsto[x_1:x_2],
$$
which is surjective, it is well-defined and $\Psi(\sph^1)=\R\PP^1$. Define 
$$
F:=[F_0:\cdots:F_m]=\Pi\circ\Psi:\{\x_1^2+\x_2^2-\x_0^2=0\}\to\C\PP^m
$$ 
and observe that $F_0=(\x_1^2+\x_2^2)^p=\x_0^{2p}$ on the set $\{\x_1^2+\x_2^2-\x_0^2=0\}$, so we may assume $F_0=\x_0^{2p}$. As $\sph^1=\{\x_1^2+\x_2^2-\x_0^2=0,\x_0=1\}$, the restriction map 
$$
f:=F|_{\sph^1}:=[1:F_1:\ldots:F_m]:\sph^1\to\R^m
$$ 
is the restriction of a polynomial map to $\sph^1$ and $f(\sph^1)=\Pi(\R\PP^1)=S$, as required.
\qed

\begin{remark}\label{degree}
The invariant regular map $\widetilde{F}:=(\widetilde{F}_0,\widetilde{F}_1):\C\PP^1\to\C\PP^1$ of \eqref{cp1cp1} satisfies
$$
\begin{cases}
\widetilde{F}_0={\tt i}\frac{\lambda_1}{2}(\x_0+{\tt i}\x_1)^{k_1}-{\tt i}\frac{\ol{\lambda_1}}{2}(\x_0-{\tt i}\x_1)^{k_1}\\
\widetilde{F}_1=\frac{\lambda_1}{2}(\x_0+{\tt i}\x_1)^{k_1}+\frac{\ol{\lambda_1}}{2}(\x_0-{\tt i}\x_1)^{k_1}
\end{cases}
$$
We have $\widetilde{F}|_{\R\PP^1}:\R\PP^1\to\R\PP^1,\ [x_0:x_1]\mapsto[-\Im(\lambda_1(x_0+{\tt i}x_1)^{k_1}):\Re(\lambda_1(x_0+{\tt i}x_1)^{k_1})]$ and we may assume that $\lambda_1\ol{\lambda_1}=1$. We refer the reader to \cite[Ch.VIII.\S2]{ma} for the concept and main properties of the {\em topological degree of a continuous map} $f:\sph^1\to\sph^1$. Consider the regular maps $\psi:\sph^1\to\R\PP^1,\ (x,y)\mapsto[x:y]$, which has topological degree $2$, $\phi:\R\PP^1\to\sph^1,\ [t_0:t_1]\mapsto(\frac{t_1^2-t_0^2}{t_0^2+t_1^2},-\frac{2t_0t_1}{t_0^2+t_1^2})$ and the composition 
\begin{multline*}
\phi\circ\widetilde{F}|_{\R\PP^1}\circ\psi:\sph^1\to\sph^1,\ (x,y)\equiv x+{\tt i}y=:z\mapsto\lambda_1^2z^{2k_1}=\lambda_1^2(x+{\tt i}y)^{2k_1}\\
\equiv(\Re(\lambda_1^2(x_0+{\tt i}x_1)^{2k_1}),\Im(\lambda_1^2(x_0+{\tt i}x_1)^{2k_1}),
\end{multline*}
which has topological degree $2k_1$. Consequently, $\widetilde{F}|_{\R\PP^1}:\R\PP^1\to\R\PP^1$ has topological degree $k_1\geq1$.
$ $\hfill$\sqbullet$
\end{remark}

\begin{proof}[Proof of Corollary \em\ref{ps1c}]
It is enough to apply Theorem \ref{ps1} for $m=2$ using the equivalence between the restrictions to $\sph^1$ of Laurent polynomials in $\C[\z,\z^{-1}]$ and polynomial maps $\R^2\to\R^2$ already seen in \S\ref{pclp}. 
\end{proof}

\begin{proof}[Proof of Proposition \em\ref{ps2}]
The implication (i) $\Longrightarrow$ (ii) is clear, so let us prove (ii) $\Longrightarrow$ (iii). Let $k\geq2$ and let $f:\R^{k+1}\to\R^m$ be a polynomial map such that $f(\sph^k)=S$. Let $F:\C^{k+1}\to\C^m$ be the (invariant) polynomial extension of $f$ to $\C^{k+1}$. Let $Y:=\{\x_1^2+\cdots+\x_{k+1}^2-\x_0^2=0\}\subset\C\PP^{k+1}$ be the Zariski closure of $\sph^k$ in $\C\PP^{k+1}$, which is a non-singular complex projective algebraic set, and let $X:=\cl_{\C\PP^m}^{\zar}(S)$, which is irreducible, because $S$ is by \cite[(3.1)(iv)]{fg3} an irreducible semialgebraic set. As $F$ is continuous for the Zariski topology, $F$ is a polynomial map and $F(\sph^k)=S$, we deduce $F(Y\cap\C^{k+1})\subset X\cap\C^m$. By Chevalley's elimination Theorem \cite[Thm.(2.31)]{m1} $F(Y\cap\C^{k+1})$ is an invariant constructible subset of $X\cap\C^m$. As $S$ has (real) dimension $1$, the intersection $X\cap\C^m$ has (complex) dimension $1$. As $X$ is irreducible, there exists an invariant (non-empty) finite set $E\subset\C\PP^m$ such that $X\cap\mathsf{H}^m_\infty(\C)\subset E$ and $F(Y\cap\C^{k+1})=X\setminus E\subset\C^m$. 

Let $H\subset\R^{k+1}$ be any $2$-dimensional plane through the origin. Then $H\cap\sph^k$ is a circle of center the origin and radius $1$ and there exists an affine isomorphism $\eta:\sph^1\to H\cap\sph^k$. As $F$ is non-constant (and for each pair of points of $\sph^k$ there exists a two dimensional plane through such points and the origin), there exists $H$ such that $f|_{H\cap\sph^k}:H\cap\sph^k\to S$ is non-constant, so $T:=f(H\cap\sph^k)$ is a one dimensional semialgebraic subset of $S$. As $S$ and $T$ are by \cite[(3.1)(iv)]{fg3} irreducible, both have dimension $1$ and $T\subset S$, we deduce $X=\cl_{\C\PP^m}^{\zar}(S)=\cl_{\C\PP^m}^{\zar}(T)$. As $T$ is the image of $\sph^1$ under $g:=f\circ\eta$, we conclude by Theorem \ref{ps1} that the Zariski closure $X=\cl_{\C\PP^m}^{\zar}(T)$ is an invariant rational curve such that one of the following three cases hold:
\begin{itemize}
\item[(1)] $\cl_{\C\PP^m}^{\zar}(T)\cap\mathsf{H}^m_\infty(\C)=\{p\}$ is a singleton (that belongs to $\mathsf{H}^m_\infty(\R)$) and the analytic set germ $\cl_{\C\PP^m}^{\zar}(T)_p$ is irreducible.
\item[(2)] $\cl_{\C\PP^m}^{\zar}(T)\cap\mathsf{H}^m_\infty(\C)=\{p\}$ is a singleton (that belongs to $\mathsf{H}^m_\infty(\R)$), the analytic set germ $\cl_{\C\PP^m}^{\zar}(T)_p$ has exactly two irreducible components that are conjugated, and $T=\cl_{\R\PP^m}^{\zar}(T)_{(1)}$.
\item[(3)] $\cl_{\C\PP^m}^{\zar}(T)\cap\mathsf{H}^m_\infty(\C)=\{q,\ol{q}\}$ (where $q,\ol{q}\not\in\mathsf{H}^m_\infty(\R)$), the analytic set germs $\cl_{\C\PP^m}^{\zar}(T)_q$ and $\cl_{\C\PP^m}^{\zar}(T)_{\ol{q}}$ are irreducible and conjugated, and $T=\cl_{\R\PP^m}^{\zar}(T)_{(1)}$.
\end{itemize}
Let us discard cases (2) and (3). In such cases $T\subset S\subset\cl_{\R\PP^m}^{\zar}(S)_{(1)}=\cl_{\R\PP^m}^{\zar}(T)_{(1)}=T$, so $T=S=\cl_{\R\PP^m}^{\zar}(S)_{(1)}$. 

Let $X:=\{\x_1^2+\x_2^2-\x_0^2=0\}\subset\C\PP^2$ and let $G:X\dasharrow\cl_{\C\PP^m}^{\zar}(T)$ be the regular extension of $g$ to $X$. Consider the parameterization $\Phi:\C\PP^1\to X,\ [\t_0:\t_1]\to[\t_0^2+\t_1^2:2\t_0\t_1:\t_1^2-\t_0^2]$ and the composition $P:=G\circ\Phi:\C\PP^1\to\cl_{\C\PP^m}^{\zar}(S)$. Let $\Pi:\C\PP^1\to X$ be an invariant normalization of $X$ and let $\widetilde{P}:\C\PP^1\to\C\PP^1$ be an invariant regular map such that $P=\Pi\circ\widetilde{P}$ (see the beginning of the proof of Theorem \ref{ps1} for further details concerning the construction of the previous regular maps). The restriction $\widetilde{P}|_{\R\PP^1}:\R\PP^1\to\R\PP^1$ is either $\id_{\R\PP^1}$ and has topological degree $1$ or $\widetilde{P}|_{\R\PP^1}\neq\id_{\R\PP^1}$ and has by Remark \ref{degree} topological degree $\geq1$.

By \S\ref{nac}($\bullet$) we have $\Pi(\R\PP^1)=\cl_{\R\PP^m}^{\zar}(S)_{(1)}=S$. As $\Pi:\C\PP^1\to X$ is a proper finite map and $E\neq\varnothing$ is invariant, we deduce that $E':=\Pi^{-1}(E)$ is also an invariant finite (non-empty) set and $\Pi:\C\PP^1\setminus E'\to X\setminus E$ is proper, finite and surjective. Thus, $\Pi|_{\C\PP^1\setminus E'}:\C\PP^1\setminus E'\to X\setminus E$ is the normalization of $X\setminus E$. As $Y\cap\C^{k+1}$ is non-singular, it is a normal affine complex algebraic set. By the universal property of normalization \cite[Ch.2.\S5.Thm.5(ii), pag.130]{sh1} there exists an invariant regular map $F^\bullet:Y\cap\C^{k+1}\to\C\PP^1\setminus E'$ such that $F|_{Y\cap\C^{k+1}}=\Pi|_{\C\PP^1\setminus E'}\circ F^\bullet$. Let $E:X\to Y$ be a regular extension of $\eta:\sph^1\to\sph^k$ and observe that $\Pi\circ\widetilde{P}=F\circ E\circ\Phi=\Pi\circ F^\bullet\circ E\circ\Phi$ outside a finite subset of $\C\PP^1$, so $\widetilde{P}=F^\bullet\circ E\circ\Phi$ outside a finite subset of $\C\PP^1$. We have $\widetilde{P}|_{\R\PP^1}=F^\bullet|_{\sph^k}\circ\eta\circ\Phi|_{\R\PP^1}:\R\PP^1\to\sph^1\to\sph^k\to\R\PP^1$. Fix a point $c\in\R\PP^1$ and consider the induced chain of homomorphisms between the homotopy groups (see \cite[Ch.II.\S4]{ma} for further details)
\begin{multline*}
(\widetilde{P}|_{\R\PP^1})_*=(F^\bullet|_{\sph^k})_*\circ\eta_*\circ(\Phi|_{\R\PP^1})_*:\\
\pi_1(\R\PP^1,c)\to\pi_1(\sph^1,\Phi(c))\to\pi_1(\sph^2,\eta(\Phi(c)))\to\pi_1(\R\PP^1,\widetilde{P}(c)).
\end{multline*}
As $\pi_1(\sph^2,\eta(\Phi(c)))=0$, we deduce $(\widetilde{P}|_{\R\PP^1})_*=0$, which is a contradiction, because $\widetilde{P}|_{\R\PP^1}$ has topological degree $\geq1$, as we have explained above.

Consequently, only case (1) is possible and assertion (iii) holds.

We finally check (iii) $\Longrightarrow$ (i). By Theorem \ref{pb1} there exists a polynomial map $g:\R\to\R^m$ such that $g([-1,1])=S$. The projection $\rho:\R^3\to\R,\ (x,y,z)\mapsto x$ satisfies $\rho(\sph^2)=[-1,1]$. Thus, the composition $f:=g\circ\rho:\R^3\to\R^m$ is a polynomial map such that $(g\circ\rho)(\sph^2)=g([-1,1])=S$, as required.
\end{proof}

\end{document}